\documentclass[12pt,a4paper]{amsart}
\usepackage[top=1.5in, bottom=1.25in, left=1.25in, right=1.25in]{geometry}

\usepackage{amsmath}
\usepackage{amssymb}
\usepackage{amsfonts}
\usepackage{palatino, mathrsfs}
\usepackage{url}
\usepackage{xcolor}
\usepackage{mathtools}
\usepackage{fullpage, microtype, colonequals}

\newtheorem{theorem}{Theorem}[section]
\newtheorem{definition}{Definition}
\newtheorem{lemma}[theorem]{Lemma}
\newtheorem{proposition}[theorem]{Proposition}

\newtheorem*{conjecture*}{Conjecture}
\newtheorem*{theorem*}{Theorem}
\newtheorem*{remark*}{Remark}
\numberwithin{equation}{section}

\begin{document}

\begin{center}

\title{Three-Dimensional stochastic Navier-Stokes equations with Markov switching}

\author{Po-Han Hsu$^\ast$}
\address{4415 French Hall-West, University of Cincinnati, Cincinnati, OH,  45221-0025, USA}
\email{hsupa@ucmail.uc.edu}

\author{Padmanabhan Sundar}
\address{316 Lockett Hall, Louisiana State University, Baton Rouge, LA, 70803-4918,  USA}
\email{psundar@lsu.edu}

\date{\today}

\thanks{$\ast$The author was partially supported by NSF grant DMS-1622026.}

\begin{abstract}
A finite-state Markov chain is introduced in the noise terms of the three-dimensional stochastic Navier-Stokes equations in order to allow for transitions between two types of multiplicative noises. We call such  systems as stochastic Navier-Stokes equations with Markov switching.  To solve such a system,  a family of regularized stochastic systems is introduced. For each such regularized system, the existence of a unique strong solution (in the sense of stochastic analysis) is established by the method of martingale problems and pathwise uniqueness. The regularization is removed in the limit by obtaining a weakly convergent  sequence from the family of regularized solutions, and identifying the limit as a solution of the three-dimensional stochastic Navier-Stokes equation with Markov switching. 
\end{abstract}
\maketitle
\end{center}
\section{Introduction}\label{intro}
Let $G$ be an open bounded domain in $\mathbb{R}^{3}$ with a smooth boundary. Let the three-dimensional vector-valued function ${\bf u}(x, t)$ and  the real-valued function $p(x, t)$ denote the velocity and pressure of the fluid at each $x\in G$ and time $t\in [0, T]$. The motion of viscous incompressible flow on $G$ with no slip at the boundary is described by the Navier-Stokes system:
\begin{align}
\partial_{t}{\bf u}-\nu\Delta{\bf u}+({\bf u}\cdot\nabla){\bf u}-\nabla p&={\bf f}(t)\quad&&\mbox{in}\quad G\times[0, T] \label{eq},
\\ \nabla\cdot {\bf u}&=0\quad&&\mbox{in}\quad G\times[0, T],\nonumber
\\ {\bf u}(x,t)&=0 \quad&&\mbox{on}\quad\partial G\times[0, T],\nonumber
\\ {\bf u}(x, 0)&={\bf u}_{0}(x)\quad&&\mbox{on}\quad G\times\{t=0\},\nonumber
\end{align}
where $\nu>0$ denotes the viscosity coefficient, and the function ${\bf f}(t)$ is an external body force.  The equation \eqref{eq} can be written in the abstract evolution form on a suitable space as follows:
\begin{align}\label{evolution B}
{\bf du}(t)+[\nu{\bf Au}(t)+{\bf B}({\bf u}(t))]dt={\bf f}(t)dt,
\end{align}
where ${\bf A}$ is the Stokes operator and ${\bf B}$ is the nonlinear inertial operator  introduced in Section \ref{pre}. 

A random body force, in the form of a multiplicative noise driven by a Wiener process  $W(t)$, is added to the model (see, e.g., \cite{turbulence}) so that one obtains
\begin{align*}
\mathbf{du}(t)+[\nu\mathbf{Au}(t)+\mathbf{B}(\mathbf{u}(t))]dt=\mathbf{f}(t)dt+\sigma(t, {\bf u}(t))dW(t).%\label{eq Wiener}
\end{align*}
Originally, it was Kolmogorov who suggested the introduction of white noise on the right side of equation \eqref{evolution B} in order to investigate the existence of invariant measures (cf. Vishik and Fursikov \cite{VF}). From then on, several works on stochastic Navier-Stokes equations appeared with an additive or more generally, a multiplicative noise driven by a Wiener process.

In addition, if the noise is allowed to be ``discontinuous,'' then a term driven by a compensated Poisson random measure $\tilde{N_{1}}(dz, ds)$ (which is independent of $W(t)$) is added so that the equation becomes
 \begin{align}\label{eq tur}
&\mathbf{du}(t)+[\nu\mathbf{Au}(t)+\mathbf{B}(\mathbf{u}(t))]dt
\\&=\mathbf{f}(t)dt+\sigma(t, {\bf u}(t))dW(t)+\int_{Z}{\bf G}(t, {\bf u}(t-), z)\tilde{N_{1}}(dz, dt),\nonumber
\end{align}
where $\tilde{N_{1}}(dz, dt):=N_{1}(dz, dt)-\nu_{1}(dz)dt$ and $\nu_{1}(dz)dt$ is the intensity measure of $N_{1}(dz, dt)$. The rationale for the presence of a discontinuous noise (driven by a Poisson random measure) in equation \eqref{eq tur} is given in Birnir \cite{turbulence}. In short, discontinuities arise from the prevalence of point vorticities in  fluid flows in turbulent regime.

Stochastic Navier-Stokes systems have been studied by a number of authors at various levels of generality. Spurred by the works of Bensoussan and Temam \cite{BT}, and Viot \cite{Vi}, there was an active growth in the area with notable contributions by Flandoli and Gatarek \cite{FG}, Flandoli and Maslowski \cite{FM}, Debussche and Da Prato \cite{DD}, Menaldi and Sritharan \cite{MS},  Mattingly \cite{Ma}, R\"ockner and Zhang \cite{RZ} and Sritharan and Sundar \cite{SS}, to name a few. The references in the articles by Flandoli  as well as Albeverio \cite{SPDE in hydrodynamic} would provide a more complete list of research work on stochastic Navier-Stokes equations. 

A {\it novelty} of this paper consists in the introduction of a right continuous  Markov chain $\{\mathfrak{r}(t): t\in\mathbb{R}^{+}\}$ in order to allow for transitions in the type of random forces that perturb the Navier-Stokes equation. The equation under study appears as
\begin{align}\label{N-S M original}
&\mathbf{du}(t)+[\nu\mathbf{Au}(t)+\mathbf{B}(\mathbf{u}(t))]dt
\\&=\mathbf{f}(t)dt+\sigma(t, \mathbf{u}(t),\mathfrak{r}(t))dW(t)
+\int_{Z}\mathbf{G}(t, \mathbf{u}(t-),\mathfrak{r}(t-), z)\tilde{N_{1}}(dz, dt)\nonumber
\end{align}
with initial condition ${\bf u}(0)={\bf u}_{0}$ in a specified space. The Markov chain is assumed to be independent of the Wiener process and the Poisson random measure, and it brings transitions between smooth (e.g., laminar) and turbulent flows into stochastic Navier-Stokes equations. We shall call such equations as \emph{stochastic Navier-Stokes equations with Markov switching}. 

The objective of this article is to construct a weak solution (in the sense of stochastic analysis and partial differential equations) to equation \eqref{N-S M original}, and it is achieved by the following steps. First, we regularize the nonlinear term in \eqref{eq} and solve the regularized equation: for each $\epsilon>0$,
\begin{align}
\partial_{t}{\bf u}-\nu\Delta{\bf u}+((k_{\epsilon}{\bf u})\cdot\nabla){\bf u}-\nabla p={\bf f}(t),\label{r eq}
\end{align}
the operator $k_{\epsilon}$ is a mollification operator (see equation \eqref{k_{epsilon}} below). The abstract evolution form of equation \eqref{r eq} is
\begin{align}\label{eq det}
\mathbf{du}(t)+[\nu\mathbf{Au}(t)+\mathbf{B}_{k_{\epsilon}}(\mathbf{u}(t))]dt%\label{eq det}
=\mathbf{f}(t)dt,
\end{align}
where the operator ${\bf B}_{k_{\epsilon}}$ will be introduced in Section \ref{pre}. Thus, its stochastic analog with Markov switching is given by
\begin{align}\label{N-S M}
&\mathbf{du}(t)+[\nu\mathbf{Au}(t)+\mathbf{B}_{k_{\epsilon}}(\mathbf{u}(t))]dt
\\&=\mathbf{f}(t)dt+\sigma(t, \mathbf{u}(t),\mathfrak{r}(t))dW(t)
+\int_{Z}\mathbf{G}(t, \mathbf{u}(t-),\mathfrak{r}(t-), z)\tilde{N_{1}}(dz, dt)\nonumber
\end{align}
with initial condition ${\bf u}(0)={\bf u}_{0}$ in a specified space. Our first objective is to show that equation \eqref{N-S M} admits a unique strong solution (in the sense of stochastic analysis) under suitable growth and Lipschitz conditions on the noise coefficients (listed later as Hypotheses {\bf H}):
\begin{theorem}\label{existence and uniqueness of N-S M}
Assume that $\mathbb{E}|{\bf u}(0)|^{3}<\infty$ and ${\bf f}\in L^{3}(0, T; V')$. Then under Hypotheses $\bf H$, there exists a unique strong solution to the stochastic system \eqref{N-S M} for each fixed $\epsilon > 0$.
\end{theorem}

Let   $\bf u^{\epsilon}$ denote the solution to the regularized equation \eqref{N-S M}. The next step is to show that there exists a sequence from the family  $\{{\bf u}^{\epsilon}\}_{\epsilon>0}$ which  converges weakly to a limit as $\epsilon \to 0$. Let the limit be denoted by  $\bf u$. This sets up the stage  to identify the limit, $\bf u$, as a solution of the stochastic Navier-Stokes equation with Markov switching \eqref{N-S M original} without regularization. The idea of using a regularization for the nonlinear term in the Navier-Stokes system goes back to Leray (see, e.g., \cite{Leray, O-P}). In its stochastic context, we have
\begin{theorem}\label{THE main theorem}
Assume that $\mathbb{E}|{\bf u}(0)|^{3}<\infty$ and ${\bf f}\in L^{3}(0, T; V')$. Then under Hypotheses $\bf H$, there exists a weak solution $\bf u$ to the three-dimensional stochastic Navier-Stokes equation with Markov switching \eqref{N-S M original}.
\end{theorem}

Ergodic behavior of the solution of the regularized stochastic Navier-Stokes system with Markov switching and its limit (along a sequence) is being prepared by us as a separate article. Coupling of the stochastic Navier-Stokes sytem and the Markov switching to reflect onset of turbulence is currently under study.

The present article is organized as follows. The background results and the functional analytic setup for the Navier-Stokes system are introduced in Section \ref{pre}. A priori estimates appear in Section \ref{a priori estimates}.  Section \ref{regularized} is devoted to the proof of Theorem \ref{existence and uniqueness of N-S M}. In section \ref{original},  the proof of Theorem \ref{THE main theorem} is presented.

\section{Preliminaries and Functional Analytic Setup}\label{pre}

\subsection{Basic Results on Convolution}
First, we recall some properties on convolution in order to explain regularization. The interested reader may consult, e.g., \cite[App. C.5.]{Evans} for more details. If $U\subset\mathbb{R}^{3}$ is open and $\epsilon>0$, we write
$
U_{\epsilon}:=\{x\in U: \text{dist}(x, \partial U)>\epsilon\}.
$
Define the function $\eta\in C^{\infty}(\mathbb{R}^{3})$ by
\begin{align*}%\label{standard mollifier}
\eta(x):=
\begin{cases}
C\exp\Big(\frac{1}{|x|^{2}-1}\Big)& \mbox{if}\quad |x|<1\\
0 & \mbox{if}\quad |x|\geq 1,
\end{cases}
\end{align*}
where the constant $C>0$ is selected so that $\int_{\mathbb{R}^{3}}\eta dx=1$. For each $\epsilon>0$, set
$
\eta_{\epsilon}(x):=\frac{1}{\epsilon^{3}}\eta\Big(\frac{x}{\epsilon}\Big).
$
We call $\eta$ the standard mollifier; the function $\eta_{\epsilon}$ is a smooth function on $\mathbb{R}^{3}$ with support in $B(0, \epsilon)$ and satisfy
$
\int_{\mathbb{R}^{3}}\eta_{\epsilon}dx=1.
$
If $f: U\rightarrow\mathbb{R}$ is locally integrable, define the mollification operator by 
\begin{align}\label{k_{epsilon}}
k_{\epsilon}f:=\eta_{\epsilon}\ast f \quad\mbox{in} \quad U_{\epsilon},
\end{align}
i.e., 
$
k_{\epsilon}f=\int_{U}\eta_{\epsilon}(x-y)f(y)dy=\int_{B(0, \epsilon)}\eta(y)f(x-y)dy
$
for $x\in U_{\epsilon}$.
The next lemma collects some properties of the mollification operator. The interested reader may consult, e.g., \cite[Thm. 7 in App. C.5]{Evans} or \cite[Lem. 6.3]{O-P} for details.
\begin{lemma}\label{convolution}
The mollification operator enjoys the following properties:
\begin{enumerate}
\item $k_{\epsilon}f\in C^{\infty}(U_{\epsilon})$.
%\item $k_{\epsilon}f\rightarrow f$ almost everywhere as $\epsilon\rightarrow 0$.
%\item If $f\in C(U)$, then $k_{\epsilon}f\rightarrow f$ uniformly on compact subsets of $U$.
\item If $1\leq p<\infty$ and $f\in L^{P}_{loc}(U)$, then $k_{\epsilon}f\rightarrow f$ in $L^{p}_{loc}(U)$.
\item If $1\leq p<\infty$ and $f\in L^{p}_{loc}(U)$, then $\|k_{\epsilon}f\|_{L^{p}_{loc}(U)}\leq \|f\|_{L^{p}_{loc}(U)}$.
\end{enumerate}
\end{lemma}

\subsection{Function Space and Operators}
Let $\mathcal{D}(G)$ be the space of $C^{\infty}$-functions with compact support contained in $G$ and $\mathcal{V}\colonequals\{{\bf u}\in\mathcal{D}(G): \nabla\cdot{\bf u}=0\}$. Let $H$ and $V$ be the completion of $\mathcal{V}$ in $L^{2}(G)$ and $W^{1, 2}_{0}(G)$, respectively. Then it can be shown that (see, e.g., \cite[Sec. 1.4, Ch. I]{Temam})
\begin{align*}
H&=\{{\bf u}\in L^{2}(G): \nabla\cdot{\bf u}=0,\ {\bf u}\cdot{\bf n}\big|_{\partial G}=0\},\\
V&=\{{\bf u}\in W^{1, 2}_{0}(G):  \nabla\cdot{\bf u}=0\},
\end{align*}
and we denote the $H$-norm  ($V$-norm, resp.) by $|\cdot|$ ($\|\cdot\|$, resp.) and the inner product on $H$ (on $V$, resp.) by $(\cdot, \cdot)$  
($((\cdot, \cdot))$, resp.). The duality pairing between $V^{\prime}$ and $V$  is denoted by $\langle\cdot,\cdot\rangle_{V}$, or simply by $\langle\cdot,\cdot\rangle$ when there is no ambiguity. In addition, we have the following inclusion between the spaces:
$
V\hookrightarrow H\hookrightarrow V',
$
and both of the inclusions $V\hookrightarrow H$ and $H\hookrightarrow V'$ are dense,  compact embeddings ( see, e.g., \cite[Lem. 1.5.1 and 1.5.2, Ch. II]{sohr}).

Let $\bf A: V\rightarrow V'$ be the Stokes operator. The set $\{e_{i}\}_{i=1}^{\infty}$ is reserved for the orthonormal basis in $H$ (orthogonal in $V$) that consists of the eigenvector of  the Stokes operator $\bf A$ (see, e.g., \cite[Thm. IV. 5.5]{BF}), and $H_n\colonequals$ span$\{e_i\}_{i=1}^{n}$; $\Pi_{n}$ is the orthogonal projection of $H$ on $H_n$. In addition, for all ${\bf u\in\mathcal{D}({\bf A})}$, one has (see, e.g., \cite[Sec. 6, Ch. II]{FMRT})
\begin{align}\label{property of Stokes operator}
\langle{\bf Au}, {\bf u}\rangle_{V}=\|{\bf u}\|.
\end{align}
Define 
$b(\cdot,\cdot,\cdot): V\times V\times V\rightarrow\mathbb{R}$ by
\begin{align*}
b({\bf u},{\bf v},{\bf w}):=\sum_{i,j=1}^{3}\int_{G}{u}_i\frac{\partial v_j}{\partial x_i}w_j dx.
\end{align*}
Then $b$ is a trilinear form which induces a bilinear form ${\bf B}({\bf u,v})$ by $b({\bf u},{\bf v,w})=\langle {\bf B}{\bf (u,v),w}\rangle_{V}$. In addition, $b$ enjoys the following properties (see, e.g., \cite[Lem. 1.3, Sec. 1, Ch. II]{Temam}):
\begin{align}
b({\bf u}, {\bf v}, {\bf v})&=0\label{vanishing of b},\\
b({\bf u}, {\bf v}, {\bf w})&=-b({\bf u, w, v}).\label{negative of b}
\end{align}

For each $\epsilon>0$, define  
$b(k_{\epsilon}\cdot,\cdot,\cdot): V\times V\times V\rightarrow\mathbb{R}$ by
\begin{align}\label{def of b_{k}}
b(k_{\epsilon}{\bf u},{\bf v},{\bf w}):=\sum_{i,j=1}^{3}\int_{G}{(\eta_{\epsilon}\ast u)}_i\frac{\partial v_j}{\partial x_i}w_j dx,
\end{align}
which induces a bilinear form ${\bf B}_{k_{\epsilon}}({\bf u,v})$ by 
$b(k_{\epsilon}{\bf u},{\bf v,w})=\langle {\bf B}_{k_{\epsilon}}{\bf (u,v),w}\rangle_{V}.$ 
The regularization rises the regularity of the first component in $b$, therefore, one may employ the  (generalized) H\"older inequality and the Young convolution inequality to deduce 
\begin{align}
|b(k_{\epsilon}{\bf u}, {\bf v}, {\bf w})|
\leq \|\eta_{\epsilon}\ast {\bf u}\|_{6}\|\nabla {\bf v}\|_{2}\|{\bf w}\|_{3}
\leq C'_{\epsilon}\|{\bf u}\|_{3}\|\nabla {\bf v}\|_{2}\|{\bf w}\|_{3}\label{ineq of b},
\end{align}
where $C'_{\epsilon}=\|\eta_{\epsilon}\|_{\frac{6}{5}}$. This together with  Sobolev embedding and interpolation inequalities further implies
\begin{align}
|b(k_{\epsilon}{\bf u,v,w})|
\leq C_{\epsilon}\|{\bf u}\|^{\frac{1}{2}}|{\bf u}|^{\frac{1}{2}}\|{\bf v}\| \|{\bf w}\|^{\frac{1}{2}}|{\bf w}|^{\frac{1}{2}}\label{b_{k}uvw}.
\end{align}
In particular, when ${\bf u=w}$, we have
\begin{align}
|b(k_{\epsilon}{\bf u,v,u})|\leq C_{\epsilon}\|{\bf u}\|\cdot|{\bf u}|\cdot\|{\bf v}\|.\label{b_{k} uvu}
\end{align}
As shall be seen later, we first work with a fixed $\epsilon$. Therefore, we shall assume that $C_{\epsilon}=1$ for the sake of simplicity.

\subsection{Noise Terms}
\begin{enumerate}
\item[(i)] Let $Q\in\mathcal{L}(H)$ be a nonnegative, symmetric, trace-class operator. Define $H_{0}\colonequals Q^{\frac{1}{2}}(H)$ with the inner product given by 
$
(u, v)_{0}:=(Q^{-\frac{1}{2}}u, Q^{-\frac{1}{2}}v)_{H}
$
for $u, v\in H_{0}$, where $Q^{-\frac{1}{2}}$ is the inverse of $Q$.
%the pseudo inverse of $Q^{\frac{1}{2}}$ in the case that $Q$ is not one-to-one (see, e.g., \cite[App. C]{concise}). 
Then it follows from \cite[Prop. C.0.3 (i)]{concise} that $(H_{0}, (\cdot, \cdot)_{0})$ is again a separable Hilbert space. 
Let $\mathcal{L}_{2}(H_{0}, H)$ denote the separable Hilbert space of the Hilbert-Schmidt operators from $H_{0}$ to $H$. Then it can be shown that (see, e.g., \cite[p. 27]{concise})
$
\|L\|_{\mathcal{L}_{2}(H_{0}, H)}=\|L\circ Q^{\frac{1}{2}}\|_{\mathcal{L}_{2}(H, H)}
$
for each $L\in \mathcal{L}_{2}(H_{0}, H)$. Moreover, we write 
$
\|L\|_{L_{Q}}=\|L\|_{\mathcal{L}_{2}(H_{0}, H)}
$
for simplicity.

Let $T>0$ be a fixed real number and $(\Omega, \mathcal{F}, \{\mathcal{F}_{t}\}_{0 le t\le T}, \mathcal{P})$ be a filtered probability space.Let $W$ be an $H$-valued Wiener process with covariance $Q$.

Let $\sigma: [0, T]\times\Omega\rightarrow\mathcal{L}_{2}(H_{0}, H)$ be jointly measurable and adapted. If  we have $\mathbb{E}\int^{T}_{0}\|\sigma(s)\|^{2}_{\mathcal{L}_{2}(H_{0}, H)}ds<\infty$, then for $t\in [0, T]$, the stochastic integral
$
\int^{t}_{0}\sigma(s)dW(s)
$
is well-defined and is an $H$-valued continuous square integrable martingale.

\item[(ii)] Let $({Z}, \mathcal{B}({ Z}))$ be a measurable space, ${\bf M}$ be the collection of all of nonnegative integer-valued measures on $({Z}, \mathcal{B}({ Z}))$, and $\mathcal{B}({\bf M})$ be the smallest $\sigma$-field on ${\bf M}$ with respect to which all
$
\eta\mapsto \eta(B)
$
are measurable, where $\eta\in{\bf M}$, $\eta(B)\in\mathbb{Z}^{+}\cup\{\infty\}$, and $B\in\mathcal{B}({ Z})$. Let
$N: \Omega\rightarrow {\bf M}$ be a Poisson random measure with intensity measure $\nu$. %\bigskip

For a Poisson random measure $N(dz, ds)$,  $\tilde{N}(dz, ds)\colonequals N(dz, ds)-\nu(dz)ds$ defines its compensation. Then it can be shown that (see, e.g, \cite[Sec. 3, Ch II.]{I-W}) $\tilde{N}(dz, ds)$ is a square integrable martingale, and for predictable $f$ such that
\begin{align*}
&\mathbb{E}\int^{t+}_{0}\int_{Z}|f(\cdot, z, s)|\nu(dz)ds<\infty,\,\,\text{then}\\
&\int^{t+}_{0}\int_{Z}f(\cdot, z, s)\tilde{N}(dz, ds)
\\&=\int^{t+}_{0}\int_{Z}f(\cdot, z, s)N(dz, ds)-\int^{t}_{0}\int_{Z}f(\cdot, z, s)\nu(dz)ds
\end{align*}
 is a well-defined $\mathcal{F}_{t}$-martingale.

\item[(iii)] Let $m\in\mathbb{N}$. Let $\{\mathfrak{r}(t): t\in\mathbb{R}^{+}\}$ be a right continuous Markov chain with generator $\Gamma=(\gamma_{ij})_{m\times m}$ taking values in $\mathcal{S}:=\{1, 2, 3, .....m\}$ such that
\begin{align*}
\mathcal{R}_{t}(i, j)&=\mathcal{R}(\mathfrak{r}(t+h)=j|\mathfrak{r}(t)=i )
\\&=\left\{\begin{array}{rcl}
\gamma_{ij}h+o(h)&\mbox{if}& i\neq j,\\
1+\gamma_{ii}h+o(h)&\mbox{if}&i= j,\\
\text{and}\,\,\gamma_{ii}&=-\sum_{i\neq j}\gamma_{ij}.
\end{array}\right.
\end{align*}

%The transition probability $\mathcal{R}_{t}(i, j)$ satisfies the Chapman-Kolmogorov equation:
%\begin{align*}
%\mathcal{R}_{t+s}(i, k)=\sum_{j=1}^{m}\mathcal{R}_{s}(i, j)\mathcal{R}_{t}(j, k).
%\end{align*}
%There exists a stationary distribution $\pi=(\pi_{1}, \cdots, \pi_{m})$ for this Markov chain $\mathfrak{r}(t)$, where $\pi_{j}$ satisfies
%\begin{align*}
%\lim_{t\rightarrow\infty}\mathcal{R}_{t}(i, j)=\pi_{j}
%\end{align*}
In addition, $\mathfrak{r}(t)$ admits the following stochastic integral representation  (see, e.g, \cite[Sec. 2.1, Ch. 2]{Skorohod}):
Let $\Delta_{ij}$ be consecutive, left closed, right open intervals of the real line each having length $\gamma_{ij}$ such that
\begin{align*}
\Delta_{12}&=[0, \gamma_{12}),\ 
\Delta_{13}=[\gamma_{12}, \gamma_{12}+\gamma_{13}),\cdots\\
\Delta_{1m}&=\Big[\sum_{j=2}^{m-1}\gamma_{1j}, \sum_{j=2}^{m}\gamma_{1j}\Big),\cdots\\
\Delta_{2m}&=\Big[\sum_{j=2}^{m}\gamma_{1j}+\sum_{j=1, j\neq 2}^{m-1}\gamma_{2j}, \sum_{j=2}^{m}\gamma_{1j}+\sum_{j=1, j\neq 2}^{m}\gamma_{2j}\Big)
\end{align*}
and so on. Define a function
$
h:\mathcal{S}\times\mathbb{R}\rightarrow\mathbb{R}
$
by
\begin{align}\label{def of h}
h( i, y)=
\begin{cases}
j-i & \mbox{if} \quad y\in\Delta_{ij},\\
0& \mbox{otherwise}.
\end{cases}
\end{align}%\label{definition of h}
Then
\begin{align}\label{markov repre}
d\mathfrak{r}(t)=\int_{\mathbb{R}}h(\mathfrak{r}(t-), y)N_{2}(dt, dy),
\end{align}
with initial condition $\mathfrak{r}(0)=\mathfrak{r}_{0}$, where $N_{2}(dt, dy)$ is a Poisson random measure with intensity measure $dt\times\mathfrak{L}(dy)$, in which $\mathfrak{L}$ is the Lebesgue measure on $\mathbb{R}$.

We assume that such a Markov chain, Wiener process, and the Poisson random measure are independent.

\end{enumerate}

\subsection{Hypotheses and Stochastic System}
The noise coefficients $\sigma: [0, T]\times H\times\mathcal{S}\rightarrow \mathcal{L}_{2}(H_{0}, H)$ and ${\bf G}: [0, T]\times H\times\mathcal{S}\times Z\rightarrow H$ are assumed to satisfy the following Hypotheses $\bf H$:
\begin{enumerate}
\item [$\bf H1$.] For all $t\in [0, T]$ and all $i\in\mathcal{S}$, there exists a constant $K>0$ such that
\begin{align*}
\|\sigma(t, {\bf u}, i)\|^{p}_{L_{Q}}\leq K(1+|{\bf u}|^{p})
\end{align*}
for  $p$ equal to $2, 3$ (growth condition on $\sigma$).
\item [$\bf H2$.] For all $t\in [0, T]$, there exists a constant $L>0$ such that for all ${\bf u, v}\in H$ and $i, j\in\mathcal{S}$
\begin{align*}
\|\sigma(t, {\bf u}, i)-\sigma(t, {\bf v}, i)\|^{2}_{L_{Q}}\leq L(|{\bf u}-{\bf v}|^{2})
\end{align*}
(Lipschitz condition on $\sigma$).
\item [$\bf H3$.] For all $t\in [0, T]$ and all $i\in\mathcal{S}$, there exist a constant $K>0$ such that
\begin{align*}
\int_{Z}|{\bf G}(t, {\bf u}, i, z)|^{p}\nu(dz)\leq K(1+|{\bf u}|^{p})
\end{align*}
for all $p=1, 2$, and $3$ (growth condition on $\bf G$).
\item [$\bf H4$.] For all $t\in[0, T]$, there exists a constant $L>0$ such that for all ${\bf u, v}\in H$ and $i, j\in\mathcal{S}$, 
\begin{align*}
\int_{Z}|{\bf G}(t, {\bf u}, i, z)-{\bf G}(t, {\bf v}, i , z)|^{2}\nu(dz)\leq L(|{\bf u}-{\bf v}|^{2})
\end{align*}
 (Lipschitz condition on $\bf G$).
\end{enumerate}

The transformation from equation \eqref{eq} to \eqref{evolution B} is sketched as follows (the transformation from \eqref{r eq} to \eqref{eq det} can be achieved in a similar manner). Invoking the Helmholtz decomposition, one decomposes the space $L^{2}(G)$ into the direct sum of $H$ and its orthogonal complement, namely, $L^{2}(G)=H\oplus H^{\perp}$. Moreover, by applying the Leray projection to each term of \eqref{eq}, one may write \eqref{eq} as \eqref{evolution B}. The interested reader is referred to, e.g., \cite{Lady, sohr} for more details. 

 Let $F: [0, T]\times V\times \mathcal{S}\rightarrow\mathbb{R}^{+}$ be a continuous function with its Fr\'echet derivatives $F_{t}$, $F_{v}$, and $F_{vv}$ are bounded and continuous. Define the operator
\begin{align}\label{operator L}
&\mathcal{L}F(t, {\bf v}, i)
\\&:=F_{t}(t, {\bf v}, i)+\langle -\nu{\bf Av}-{\bf B}({\bf v})+{\bf f}(t), F_{v}(t, {\bf v}, i)\rangle_{V}\nonumber
\\& +\sum_{j=1}^{m}\gamma_{ij} F(t, {\bf v}, j)\nonumber
+\frac{1}{2}tr\Big(F_{vv}(t, {\bf v}, i)\sigma(t, {\bf v}, i)Q\sigma^{\ast}(t, {\bf v}, i)\Big)
\\&+\int_{Z}\Big(F(t, {\bf v}+{\bf G}(t, {\bf v}, i, z), i)- F(t, {\bf v}, i)\nonumber
\\&\qquad\qquad-\Big(F_{v}(t, {\bf v}, i), {\bf G}(t, {\bf v}, i, z)\Big)_{H}\Big)\nu_{1}(dz).\nonumber
\end{align}
Then we have the following change of variables formula due to It\^o (see, e.g., \cite[Lem. 3 in Sec. 2.1, Ch. 2]{Skorohod}):
\begin{align*}
&F(t, {\bf u}(t), \mathfrak{r}(t))
\\&=F(0, {\bf u}(0), \mathfrak{r}(0))+\int^{t}_{0}\mathcal{L}F(s, {\bf u}(s), \mathfrak{r}(s))ds
\\&+\int^{t}_{0}\langle F_{x}(s, {\bf u}(s), \mathfrak{r}(s)), \sigma(s, {\bf u}(s), \mathfrak{r}(s))dW(s)\rangle
\\&+\int^{t}_{0}\int_{Z}\Big(F(s, {\bf u}(s-)+{\bf G}(s, {\bf u}(s-), \mathfrak{r}(s-), z), \mathfrak{r}(s-))
\\&\qquad\qquad-F(s, {\bf u}(s-),\mathfrak{r}(s-))\Big)\tilde{N}_{1}(dz, ds)
\\&+\int^{t}_{0}\int_{\mathbb{R}}\Big(F(s, {\bf u}(s-), \mathfrak{r}(s-)+h(\mathfrak{r}(s-),y))
\\&\qquad\qquad-F(s, {\bf u}(s-),\mathfrak{r}(s-))\Big)\tilde{N}_{2}(ds,dy),
\end{align*}
where $\tilde{N}_{1}(dz, ds)$ is the compensated Poisson random measure introduced earlier; 
$
\tilde{N}_{2}(ds, dy):=N_{2}(ds, dy)-\mathfrak{L}(dy)ds
$
where $N_{2}(ds, dy)$ and $\mathfrak{L}(dy)ds$ are defined  in \eqref{markov repre}; the function $h(s, y)$ is defined as in \eqref{def of h}. In particular, if $F(t, {\bf u}(t), i)=|{\bf u}(t)|^{2}$, then $\sum_{j=1}^{m}\gamma_{ij}|{\bf u}(t)|^{2}=0$. We  therefore obtain the following energy equality:
\begin{align}
&|{\bf u}(t)|^{2}\label{energy equality}
\\&=|{\bf u}(0)|^{2}+2\int^{t}_{0}\langle -\nu{\bf Au}(s)-{\bf B}({\bf u}(s))+{\bf f}(s), {\bf u}(s)\rangle_{V}ds\nonumber
\\&+\int^{t}_{0}\|\sigma(s, {\bf u}(s),\mathfrak{r}(s))\|^{2}_{L_{Q}}ds+2\int^{t}_{0}\langle {\bf u}(s), \sigma(s, {\bf u}(s),\mathfrak{r}(s))dW(s)\rangle\nonumber
\\&+\int^{t}_{0}\int_{Z}\Big(|{\bf u}(s-)+{\bf G}(s, {\bf u}(s-), \mathfrak{r}(s-), z)|^{2}-|{\bf u}(s-)|^{2}\Big)\tilde{N}_{1}(dz, ds)\nonumber
\\&+\int^{t}_{0}\int_{Z}\Big(|{\bf u}(s)+{\bf G}(s, {\bf u}(s), \mathfrak{r}(s), z)|^{2}-|{\bf u}(s)|^{2}\nonumber
\\&\qquad\qquad-2\Big({\bf u}(s), {\bf G}(s, {\bf u}(s), \mathfrak{r}(s), z)\Big)_{H}\Big)\nu_{1}(dz)ds.\nonumber
\end{align}

\subsection{Path Space and its Topology}
Denoted by $\{\tau_{i}\}_{i=1}^{4}$ the topologies
\begin{alignat*}{2}
\tau_{1}&=\mbox{$J$-topology}\quad  &&\mbox{on} \quad \mathcal{D}([0, T]; V'),\\
\tau_{2}&=\mbox{weak topology}\quad &&\mbox{on}\quad L^{2}(0, T; V),\\
\tau_{3}&=\mbox{weak-star topology}\quad &&\mbox{on}\quad L^{\infty}(0, T; H),\\
\tau_{4}&=\mbox{strong topology}\quad&&\mbox{on}\quad L^{2}(0, T; H),
\end{alignat*}
and $\Omega_{i}$ the spaces
\begin{align*}
\Omega_{1}&=\mathcal{D}([0, T]; V'),\\
\Omega_{2}&=L^{2}(0, T; V),\\
\Omega_{3}&=L^{\infty}(0, T; H),\\
\Omega_{4}&=L^{2}(0, T; H).
\end{align*}
Then $\{(\Omega_{i}, \tau_{i})\}_{i=1}^{4}$ are all Lusin spaces (a topological space that is homeomorphic to a Borel set of a Polish space).

\begin{definition}\label{the path space of u}
Define the space $\Omega^{\ast}$ by
$
\Omega^{*}=\cap_{i=1}^{4}\Omega_{i}.
$
Let $\tau$ be the supremum of the topologies\footnote{The coarest topology that is finer than each $\tau_{i}$. See, e.g., \cite[Sec. 5.2]{Howes}} induced on $\Omega^{\ast}$ by all $\tau_{i}$. Then it follows from a result of Metivier \cite[Prop. 1, Ch. IV]{Metivier} that\footnote{Note that all the natural inclusion $\Omega_{i}\hookrightarrow\Omega_{1}$, $i=2, 3, 4$, are continuous.}
\begin{enumerate}
\item $(\Omega^{\ast}, \tau)$ is a Lusin space.
\item Let $\{\mu_{k}\}_{k\in\mathbb{N}}$ be a sequence of Borel probability laws on $\Omega^{\ast}$ (on the Borel $\sigma$-algebra $\mathcal{B}(\tau)$) such that their images $\{\mu^{i}_{k}\}_{k\in\mathbb{N}}$ on $(\Omega_{i}, \mathcal{B}(\tau_{i}))$ are tight for $\tau_{i}$ for all $i$. Then $\{\mu_{k}\}_{k\in\mathbb{N}}$ is tight for $\tau$. 
\end{enumerate}
\end{definition}
Let $(\Omega, \mathcal{F}, \mathcal{P})$ be a (complete) probability space on which the following are defined:
\begin{enumerate}
\item $W=\{W(t): 0\leq t\leq T\}$, an $H$-valued $Q$-Wiener process.
\item $N=\{N(z, t): 0\leq t\leq T\quad\mbox{and}\quad z\in Z\}$, the Poisson random measure.
\item $\mathfrak{r}=\{\mathfrak{r}(t): 0\leq t\leq T\}$, the Markov chain.
\item $\xi$, an $H$-valued random variable.
%\item $r$, an integer-valued random variable.
\end{enumerate}
Assume that $\xi$, $W$, $N$, and $\mathfrak{r}$ are mutually independent. For each $t$, define the $\sigma$-field $\mathcal{F}_{t}$ to be
$
\sigma(\xi, \mathfrak{r}(t), W(s), N(z, s): z\in Z, 0\leq s\leq t)\ \cup\ \{\text{all $\mathcal{P}$-null sets in $\mathcal{F}$}\}.
$
Then it is clear that $(\mathcal{F}_{t})$ satisfies the usual conditions, and both $W(t)$ and $N(z, t)$ are $\mathcal{F}_{t}$-adapted processes.

Denoting by $\mathcal{J}$ the $J$-topology in the space $\mathcal{D}([0, T]; S)$,  the path space of solutions of \eqref{N-S M original} is given by 
\begin{align}\label{the path space}
\begin{split}
\Omega^{\dagger}&:=\Omega^{\ast}\times\mathcal{D}([0, T]; \mathcal{S}),\\
\tau^{\dagger}&:=\tau\times\mathcal{J}.
\end{split}
\end{align}

\begin{definition}[Weak solution]\label{weak solution}
Suppose that, on some probability space $(\Omega, \mathcal{F}, \mathcal{P})$, there exists an increasing family $(\mathcal{G}_{t})$ of sub $\sigma$-field of $\mathcal{F}$, an $H\times\mathcal{S}$-valued random vector $(\xi, r)$ with given distribution $\mu$, and $\mathcal{G}_{t}$-adapted processes $W(t)$, $\tilde{N}(z, t)$, ${\bf u}(t)$, and $\mathfrak{r}(t)$ such that
\begin{enumerate}
\item $(W(t), \mathcal{G}_{t}, \mathcal{P})$ is an $H$-valued $Q$-Wiener martingale.
\item $\tilde{N}(z, t)$ is an square integrable martingale with respect to $(\mathcal{G}_{t})$.
\item $W(t)$, $N(z, t)$, and $(\xi, r)$ are mutually independent.
\item For all $\rho\in V$, 
$$\mathcal{P}\Big\{\omega: \int^{T}_{0}\langle\nu{\bf A}{\bf u}(\omega, s)+{\bf B}({\bf u}(\omega, s)),\rho\rangle_{V}ds<\infty\Big\} =1.$$
\item For all $t$ and $\rho\in V$,
\begin{align*}
\begin{split}
&\langle {\bf u}(\omega, t),\rho\rangle_{V}+\int^{t}_{0}\langle\nu{\bf A}{\bf u}(\omega, s)+{\bf B}({\bf u}(\omega, s)),\rho\rangle_{V}ds
\\&=\langle \xi,\rho\rangle_{V}+\int^{t}_{0}\langle{\bf f}(s),\rho\rangle_{V}ds+\langle\int^{t}_{0}\sigma(t, {\bf u}(s), \mathfrak{r}(s))dW(\omega, s),\rho\rangle_{V}
\\&\quad+\langle\int^{t}_{0}\int_{Z}{\bf G}(s, {\bf u}(s-), \mathfrak{r}(s-), z)\tilde{N}_{1}(\omega, dz, ds),\rho\rangle_{V}
\end{split}
\end{align*}
$\mathcal{P}$-almost surely.
\end{enumerate}
Then the family $(\Omega, \mathcal{F}, (\mathcal{G}_{t}), \mathcal{P}, \xi, r, \{W(t)\}, \{N(z, t)\}, \{{\bf u}(t)\}, \{\mathfrak{r}(t)\})$ is called a weak solution of the stochastic Navier-Stokes equation \eqref{N-S M original}.
\end{definition}

\begin{definition}[Pathwise uniqueness]
A weak solution of the stochastic Navier-Stokes equation \eqref{N-S M original} is said to be pathwise unique if, for any two weak solutions give by 
\begin{align*}
(\Omega, \mathcal{F}, (\mathcal{G}^{i}_{t}), \mathcal{P}, \xi^{i}, r^{i}, \{W(t)\}, \{N(z, t)\}, \{{\bf u}^{i}(t)\}, \{\mathfrak{r}^{i}(t)\})
\end{align*}
for $i=1, 2$, the following holds:
\begin{align*}
\mathcal{P}\Big\{({\bf u}^{1}(t), \mathfrak{r}^{1}(t))=({\bf u}^{2}(t), \mathfrak{r}^{2}(t))\quad\forall t\geq 0\Big\}=1.
\end{align*}
\end{definition}

\begin{remark*}
The definition of weak solutions to equation \eqref{N-S M} follow from Definitiion \ref{weak solution} with ${\bf B}_{k_{\epsilon}}$ in the place of $\bf B$.
\end{remark*}

%\subsection{Auxiliary lemmata}
We collect some lemmata here for the benefit of the reader.

\begin{lemma}\label{prop conv}
Let $f\in C(\Omega)$ and $\sup_{n}\mathbb{E}^{\mathcal{P}_{n}}[|f|^{1+\delta}]\leq C$ for some $\delta>0$. Let $\{\mathcal{P}_{n}\}$ be a sequence of probability measures on $\Omega$ with $\mathcal{P}_{n}\Rightarrow \mathcal{P}$, as $n\rightarrow\infty$. Then we have $\mathbb{E}^{\mathcal{P}_{n}}(|f|)\rightarrow\mathbb{E}^{\mathcal{P}}(|f|)$.
\end{lemma}

\begin{lemma}\label{Metivier cpt}
Consider the continuous dense embeddings $V\hookrightarrow H\hookrightarrow V'$ with $V\hookrightarrow H$ and $H\hookrightarrow V'$ being compact. Suppose that a set $B$ in $L^{q}(0, T; H)\cap\mathcal{D}([0, T]; V')$ is relatively compact in $\mathcal{D}([0, T]; V')$ and bounded in $L^{q}(0, T; V)$. Then $B$ is relatively compact in $L^{q}(0, T; H)$.
\end{lemma}

\begin{proof}
The proof is based on the Aubin-Lions Lemma, and we refer the interested reader to \cite[Lem. 3, Ch. VI]{Metivier}. 
\end{proof}

\begin{lemma}[Aldous' criterion]\label{Aldous}
Let $\{X_{n}\}_{n=1}^{\infty}$ be a sequence of processes with paths in the space $\mathcal{D}([0, T]; V')$. Suppose that for each rational numbers $t\in [0, T]$, we have
\begin{align}\label{Sundar Aldous}
\lim_{N\rightarrow\infty}\limsup_{n}\mathcal{P}\Big(\|X_{n}(t)\|_{V'}>N\Big)=0.
\end{align}
Then $\{X_{n}\}_{n=1}^{\infty}$ is tight in $\mathcal{D}([0, T]; V')$ if the following condition is satisfied:

For every sequence $(T_{n}, \delta_{n})$ where each $T_{n}$ is a stopping time such that $T_{n}+\delta_{n}\leq T$, and $\delta_{n}>0$, $\delta_{n}\rightarrow 0$, we have 
$
\|X_{n}(T_{n}+\delta_{n})-X_{n}(T_{n}))\|_{V'}\rightarrow 0
$
in probability as $n\rightarrow\infty$.
\end{lemma}

\begin{proof}
The interested reader may consult, e.g., \cite[p. 353 Thm. 6.8]{Walsh} for the proof of this lemma.
\end{proof}

\section{A priori estimates}\label{a priori estimates}
In the section, we establish a priori estimates to the approximation of the regularized equation \eqref{N-S M}. 
Recalling the definitions of $\{e_{i}\}_{i=1}^{\infty}$, $H_{n}$, and $\Pi_{n}$ in Section \ref{pre}, we define $W_n:=\Pi_n W$, $\sigma_n:=\Pi_n \sigma$, and ${\bf G}_n:=\Pi_n {\bf G}$. Let ${\bf u}_{0}$ be an $H$-valued random variable and $\epsilon>0$ be \emph{fixed} throughout this section.

Let ${\bf u}_{n}$ be the solution to following equation: for each ${\bf v}\in H_n$, 
\begin{align}%\label{Galerkin M}
{\bf d}({\bf u}_{n}(t), {\bf v})=&\big[\big(-\nu{\bf Au}_{n}(t)-{\bf B}_{k_{\epsilon}}({\bf u}_{n}(t)), {\bf v}\big)\big]dt+\langle{\bf f}(t),{\bf v}\rangle_{V}dt\nonumber
\\&+\big({ \sigma}_n(t,{\bf u}_{n}(t),\mathfrak{r}(t))dW_{n}(t),{\bf v}\big)\label{Galerkin M}
\\&+\big(\int_{Z}{\bf G}_n(t, {\bf u}_{n}(t-),\mathfrak{r}(t-), z)\tilde{N}_{1}(dz,dt),{\bf v}\big)\nonumber
\end{align}
with ${\bf u}_{n}(0)=\Pi_n {\bf u}_{0}$, where $\tilde{N}_{1}(dz,ds)=N_{1}(dz,ds)-\nu_{1}(dz)ds$.

\begin{proposition}[A priori estimates]\label{A priori estimates of N-S M}
Let $T>0$ be fixed. Suppose that $\mathbb{E}|{\bf u}_{0}|^{2}<\infty$ and ${\bf f}\in L^{2}(0, T; V')$. Then under Hypotheses ${\bf H}$, there exist  constants $C_1$ and $C_2$ that depend on $\mathbb{E}|{\bf u}_{0}|^{2},\mathbb{E}\int^{T}_{0}\|{\bf f}(s)\|^{2}_{V'}ds, \nu, K, T$, and $C_3$ that depends on $\mathbb{E}|{\bf u}_{0}|^{3},\mathbb{E}\int^{T}_{0}\|{\bf f}(s)\|^{3}_{V'}ds, \nu, K, T$ such that $\forall t \in [0, T]$, 
\begin{align}
\mathbb{E}|{\bf u}_{n}(t)|^{2}+\nu\mathbb{E}\int^{t}_{0}\|{\bf u}_{n}(s)\|^{2}ds &\leq C_1,\,\,\text{and}\label{L^2 M} \\
\mathbb{E}\sup_{0\leq t\leq T}|{\bf u}_{n}(t)|^{2}+\nu\mathbb{E}\int^{T}_{0}\|{\bf u}_{n}(s)\|^{2}ds&\leq C_2 \label{L^2 sup M}
\end{align}

Suppose further that $\mathbb{E}|{\bf u}_{0}|^{3}<\infty$ and $\mathbf{f}\in L^{3}(0, T; V')$. Then under Hypotheses ${\bf H}$, we have
\begin{equation}
\mathbb{E}\sup_{0\leq t\leq T}|{\bf u}_{n}(t)|^{3}+2\nu\mathbb{E}\int^{T}_{0}|{\bf u}_{n}(s)|\|{\bf u}_{n}(s)\|^{2}ds\leq C_3\label{L^3 sup M}
\end{equation}
\end{proposition}

\begin{remark*}
The  constants $C_i$ in \eqref{L^2 M}, \eqref{L^2 sup M}, and \eqref{L^3 sup M} is \emph{independent} of $\epsilon$.
\end{remark*}

\begin{proof}
Let $N>0$. Define 
\begin{align*}
\tau_{N}:=\inf\{t\in[0, T]:&|{\bf u}_{n}(t)|^{2}+\int^{t}_{0}\|{\bf u}_{n}(s)\|^{2}ds>N  
\\ \mbox{or}\ &|{\bf u}_{n}(t-)|^{2}+\int^{t}_{0}\|{\bf u}_{n}(s)\|^{2}ds>N\}.
\end{align*}
It follows from the It\^o formula that
\begin{align}\label{Ito u_{n}}
&|{\bf u}_{n}(t\wedge\tau_{N})|^{2}
\\&=|{\bf u}_{n}(0)|^{2}+2\int^{t\wedge\tau_{N}}_{0}\langle -\nu{\bf Au}_{n}(s)-{\bf B}_{k_{\epsilon}}({\bf u}_{n}(s))+{\bf f}(s), {\bf u}_{n}(s)\rangle_{V}ds\nonumber
\\&+\int^{t\wedge\tau_{N}}_{0}\|\sigma_{n}(s, {\bf u}_{n}(s),\mathfrak{r}(s))\|^{2}_{L_{Q}}ds\nonumber
\\&+2\int^{t\wedge\tau_{N}}_{0}\langle {\bf u}_{n}(s), \sigma_{n}(s, {\bf u}_{n}(s),\mathfrak{r}(s))dW_{n}(s)\rangle\nonumber
\\&+\int^{t\wedge\tau_{N}}_{0}\int_{Z}\Big(|{\bf u}_{n}(s-)+{\bf G}_{n}(s, {\bf u}_{n}(s-), \mathfrak{r}(s-), z)|^{2}\nonumber
\\&\qquad\qquad\qquad-|{\bf u}_{n}(s-)|^{2}\Big)\tilde{N}_{1}(dz, ds)\nonumber
\\&+\int^{t\wedge\tau_{N}}_{0}\int_{Z}\Big(|{\bf u}_{n}(s)+{\bf G}_{n}(s, {\bf u}_{n}(s), \mathfrak{r}(s), z)|^{2}-|{\bf u}_{n}(s)|^{2}\nonumber
\\&\qquad\qquad\qquad-2\big({\bf u}_{n}(s), {\bf G}_{n}(s, {\bf u}_{n}(s), \mathfrak{r}(s), z)\big)_{H}\Big)\nu_{1}(dz)ds.\nonumber
\end{align}
By \eqref{property of Stokes operator}, we have
\begin{align*}
2\int^{t\wedge\tau_{N}}_{0}-\nu\langle{\bf Au}_{n}(s), {\bf u}_{n}(s)\rangle_{V}ds=-2\nu\int^{t\wedge\tau_{N}}_{0}\|{\bf u}_{n}(s)\|^{2}ds.
\end{align*}
By \eqref{vanishing of b}, $\langle{\bf B}_{k_{\epsilon}}({\bf u}_{n}(s)), {\bf u}_{n}(s)\rangle_{V}=0$. For the external force term ${\bf f}$, one deduces from the basic Young inequality that
\begin{align*}
\Big|2\int^{t\wedge\tau_{N}}_{0}\langle{\bf f}(s), {\bf u}_{n}(s)\rangle_{V}\Big|
\leq\frac{1}{\nu}\int^{t}_{0}\|{\bf f}(s)\|^{2}_{V'}ds+\nu\int^{t\wedge\tau_{N}}_{0}\|{\bf u}_{n}(s)\|^{2}ds.
\end{align*}
%\begin{align*}
%&\Big|2\int^{t\wedge\tau_{N}}_{0}\langle{\bf f}(s), {\bf u}_{n}(s)\rangle_{V}ds\Big|
%\leq2\int^{t\wedge\tau_{N}}_{0}\|{\bf f}(s)\|_{V'}\|{\bf u}_{n}(s)\|ds
%\\&\leq\frac{1}{\nu}\int^{t}_{0}\|{\bf f}(s)\|^{2}_{V'}ds+\nu\int^{t\wedge\tau_{N}}_{0}\|{\bf u}_{n}(s)\|^{2}ds.
%\end{align*}
A simplification of the last term in \eqref{Ito u_{n}} gives
\begin{align*}
%&\int^{t\wedge\tau_{N}}_{0}\int_{Z}\Big(|{\bf u}_{n}(s)+{\bf G}_{n}(s, {\bf u}_{n}(s), \mathfrak{r}(s), z)|^{2}-|{\bf u}_{n}(s)|^{2}
%\\&\qquad\qquad\quad-2\big({\bf u}_{n}(s), {\bf G}_{n}(s, {\bf u}_{n}(s), \mathfrak{r}(s), z)\big)_{H}\Big)\nu_{1}(dz)ds
%\\&=
\int^{t\wedge\tau_{N}}_{0}\int_{Z}|{\bf G}_{n}(s, {\bf u}_{n}(s), \mathfrak{r}(s), z)|^{2}\nu_{1}(dz)ds.
\end{align*}
Now, taking expectation on the both side of \eqref{Ito u_{n}}, using Hypotheses ${\bf H1}$ and ${\bf H3}$, and then putting everything together, we obtain
%\begin{align*}%\label{a priori_1001}
%\begin{split}
%&\mathbb{E}|{\bf u}_{n}(t\wedge\tau_{N})|^{2}+\nu\mathbb{E}\int^{t\wedge\tau_{N}}_{0}\|{\bf u}_{n}(s)\|^{2}ds
%\\&\leq\mathbb{E}|{\bf u}_{n}(0)|^{2}+\frac{1}{\nu}\mathbb{E}\int^{T}_{0}\|{\bf f}(s)\|_{V'}^{2}ds
%\\&\quad+2K\mathbb{E}\int^{t}_{0}|{\bf u}_{n}(s\wedge\tau_{N})|^{2}ds+2KT(1+m^{2}). %\label{a priori_1001}
%\end{split}
%\end{align*}
%Recalling that  ${\bf u}_{n}(0)=\Pi_n {\bf u}^{\epsilon}_{0}$ and that ${\bf u}^{\epsilon}_{0}:=k_{\epsilon}{\bf u}_{0}$, we apply the property of projection operator $\Pi_{n}$ and (3) of Lemma \ref{convolution} to deduce
%$
%\mathbb{E}|{\bf u}_{n}(0)|^{2}\leq\mathbb{E}|{\bf u}_{0}|^{2}.
%$
%Thus, we obtain
\begin{align}\label{a priori_1001}
&\mathbb{E}|{\bf u}_{n}(t\wedge\tau_{N})|^{2}+\nu\mathbb{E}\int^{t\wedge\tau_{N}}_{0}\|{\bf u}_{n}(s)\|^{2}ds
\\&\leq\mathbb{E}|{\bf u}_{0}|^{2}+\frac{1}{\nu}\mathbb{E}\int^{T}_{0}\|{\bf f}(s)\|_{V'}^{2}ds+2K\mathbb{E}\int^{t}_{0}|{\bf u}_{n}(s\wedge\tau_{N})|^{2}ds+2KT. \nonumber
\end{align}

Denoting $C_{T}:=\mathbb{E}|{\bf u}_{0}|^{2}+\frac{1}{\nu}\mathbb{E}\int^{T}_{0}\|{\bf f}(s)\|_{V'}^{2}ds+2KT(1+m^{2})$, we utilize the Gronwall inequality to obtain
\begin{align}
\mathbb{E}|{\bf u}_{n}(t\wedge\tau_{N})|^{2}\leq C_{T}e^{2KT}.\label{a priori_1002}
\end{align}
A combination of \eqref{a priori_1001} and \eqref{a priori_1002} yield
\begin{align}
\mathbb{E}|{\bf u}_{n}(t\wedge\tau_{N})|^{2}+\nu\mathbb{E}\int^{t\wedge\tau_{N}}_{0}\|{\bf u}_{n}(s)\|^{2}ds\leq C_{T}(1+2KTe^{2KT}).\label{a priori_1004}
\end{align}
In particular,
\begin{align}\label{a priori_1003}
\mathbb{E}\int^{t\wedge\tau_{N}}_{0}|{\bf u}_{n}(s)|^{2}ds
\leq\mathbb{E}\int^{t\wedge\tau_{N}}_{0}\|{\bf u}_{n}(s)\|^{2}ds\leq\frac{1}{\nu}C_{T}(1+2KTe^{2KT}).
\end{align}

An application of Davis inequality and basic Young inequality yield
\begin{align*}
&2\mathbb{E}\sup_{0\leq t\leq T\wedge\tau_{N}}\Big|\int^{t}_{0}\langle{\bf u}_{n}(s), \sigma(s, {\bf u}_{n}(s), \mathfrak{r}(s))dW_{n}(s)\rangle\Big|
\\&\leq 2\sqrt{2}\mathbb{E}\Big\{\Big(\int^{T\wedge\tau_{N}}_{0}|{\bf u}_{n}(s)|^{2}\|\sigma(s, {\bf u}_{n}(s), \mathfrak{r}(s))\|^{2}_{L_{Q}}\Big)^{\frac{1}{2}}\Big\}
\\&\leq 2\sqrt{2}\epsilon_{1}\mathbb{E}\sup_{0\leq t\leq T\wedge\tau_{N}}|{\bf u}_{n}(t)|^{2}
+2\sqrt{2}C_{\epsilon_{1}}\mathbb{E}\int^{T\wedge\tau_{N}}_{0}\|\sigma(s, {\bf u}_{n}(s),\mathfrak{ r}(s))\|^{2}_{L_{Q}}ds.
\end{align*} 
%In conclusion, we have
%\begin{align*}
%&2\mathbb{E}\sup_{0\leq t\leq T\wedge\tau_{N}}\Big|\int^{t}_{0}\langle{\bf u}_{n}(s), \sigma(s, {\bf u}_{n}(s), \mathfrak{r}(s))dW_{n}(s)\rangle\Big|
%\\&\leq 2\sqrt{2}\epsilon_{1}\mathbb{E}\sup_{0\leq t\leq T\wedge\tau_{N}}|{\bf u}_{n}(t)|^{2}
%\\&\quad+2\sqrt{2}C_{\epsilon_{1}}\mathbb{E}\int^{T\wedge\tau_{N}}_{0}\|\sigma(s, {\bf u}_{n}(s),\mathfrak{ r}(s))\|^{2}_{L_{Q}}ds.
%\end{align*}
In a similar manner, we obtain
\begin{align*}
&2\mathbb{E}\sup_{0\leq t\leq T\wedge\tau_{N}}\int^{t}_{0}\int_{Z}\big({\bf u}_{n}(s-),  {\bf G}(s, {\bf u}_{n}(s-), \mathfrak{r}(s-), z )\big)_{H}\tilde{N}_{1}(dz, ds)
\\&\leq 2\sqrt{10}\mathbb{E}\int^{T\wedge\tau_{N}}|{\bf u}_{n}(s)||{\bf G}(s, {\bf u}_{n}(s), \mathfrak{r}(s), z)|\nu_{1}(dz)ds
\\&\leq 2\sqrt{10}\epsilon_{2}\mathbb{E}\sup_{0\leq t\leq T\wedge\tau_{N}}|{\bf u}_{n}(t)|^{2}
+2\sqrt{10}C_{\epsilon_{2}}\mathbb{E}\int^{T\wedge\tau_{N}}_{0}\int_{Z}|{\bf G}(s, {\bf u}_{n}(s), \mathfrak{r}(s), z)|^{2}\nu_{1}(dz)ds.
\end{align*}
Therefore, It\^o formula implies
\begin{align*}
&\mathbb{E}\sup_{0\leq t\leq T\wedge\tau_{N}}|{\bf u}_{n}(t)|^{2}+\nu\mathbb{E}\int^{T\wedge\tau_{N}}_{0}\|{\bf u}_{n}(s)\|^{2}ds
\\&\leq \mathbb{E}|{\bf u}_{0}|^{2}+\frac{1}{\nu}\mathbb{E}\int^{T}_{0}\|{\bf f}(s)\|_{V'}^{2}ds
+2(\sqrt{2}\epsilon_{1}+\sqrt{10}\epsilon_{2})\mathbb{E}\sup_{0\leq t\leq T\wedge\tau_{N}}|{\bf u}_{n}(t)|^{2}
\\&\quad+(2\sqrt{2}C_{\epsilon_{1}}+1)\mathbb{E}\int^{T\wedge\tau_{N}}_{0}\|\sigma(s, {\bf u}_{n}(s), \mathfrak{r}(s))\|^{2}_{L_{Q}}ds
\\&\quad+(2\sqrt{10}C_{\epsilon_{2}}+1)\mathbb{E}\int^{T\wedge\tau_{N}}_{0}\int_{Z}|{\bf G}(s, {\bf u}_{n}(s), \mathfrak{r}(s), z)|^{2}\nu(dz)ds.
\end{align*}
Take $\epsilon_{1}=\frac{1}{8\sqrt{2}}$ and $\epsilon_{2}=\frac{1}{8\sqrt{10}}$. Then $C_{\epsilon_{1}}=2\sqrt{2}$ and $C_{\epsilon_{2}}=2\sqrt{10}$. One obtains  from above and \eqref{a priori_1003} that
\begin{align}
&\frac{1}{2}\mathbb{E}\sup_{0\leq t\leq T\wedge\tau_{N}}|{\bf u}_{n}(t)|^{2}+\nu\mathbb{E}\int^{T\wedge\tau_{N}}_{0}\|{\bf u}_{n}(s)\|^{2}ds\nonumber
\\&\leq\mathbb{E}|{\bf u}_{0}|^{2}+\frac{1}{\nu}\mathbb{E}\int^{T}_{0}\|{\bf f}(s)\|_{V'}^{2}ds
+50K\mathbb{E}\int^{T\wedge\tau_{N}}_{0}(1+|{\bf u}_{n}(s)|^{2})ds\nonumber
\\&\leq\mathbb{E}|{\bf u}_{0}|^{2}+\frac{1}{\nu}\mathbb{E}\int^{T}_{0}\|{\bf f}(s)\|^{2}_{V'}ds+\frac{50K}{\nu}C_{T}(1+2KTe^{2KT})
+50KT
:=C_{2}(T),\label{L^{2} sup M for u_{n}}
\end{align}
which implies that $\tau_{N}\wedge T\rightarrow T$ as $N\rightarrow\infty$. Letting $N\rightarrow\infty$ in \eqref{L^{2} sup M for u_{n}}  and \eqref{a priori_1004}, we obtain \eqref{L^2 sup M} and \eqref{L^2 M}.
%Now consider the event $\{\tau_{N}<T\}$. By the definition of $\tau_{N}$, we have
%\begin{align*}
%\{\tau_{N}<T\}\subset\Big\{\sup_{0\leq t\leq T\wedge\tau_{N}}|{\bf u}_{n}(t)|^{2}+\int^{T\wedge\tau_{N}}_{0}\|{\bf u}_{n}(s)\|^{2}ds>N\Big\}
%\end{align*}

%Using \eqref{L^{2} sup M for u_{n}} in above, we obtain
%\begin{align*}
%\mathcal{P}\{\tau_{N}<T\}&\leq\frac{2}{N}\Big(\mathbb{E}\sup_{0\leq t\leq T\wedge\tau_{N}}|{\bf u}_{n}(t)|^{2}+\mathbb{E}\int^{T\wedge\tau_{N}}_{0}\|{\bf u}_{n}(s)\|^{2}ds\Big)
%\\&=\frac{2}{N}C_{2}(T).
%\end{align*}
%Note that $C_{2}(T)$ is a constant independent of $N$, therefore, $\mathcal{P}\{\tau_{N}<T\}\rightarrow 0$ as $N\rightarrow\infty$, which implies that $\tau_{N}\rightarrow\infty$ almost surely as $N\rightarrow\infty$. This together with the fact that $\tau_{N}$ is increasing in $N$ further imply that $T\wedge\tau_{N}\rightarrow T$ almost surely as $N\rightarrow\infty$.  Letting $N\rightarrow\infty$ in \eqref{L^{2} sup M for u_{n}}  and \eqref{a priori_1004}, we obtain \eqref{L^2 sup M} and \eqref{L^2 M}.

Define 
\begin{align*}
\tau'_{N}:=\inf\{t\in[0, T]: &|{\bf u}_{n}(t)|^3+\int^{t}_{0}|{\bf u}_{n}(s)|\|{\bf u}_{n}(s)\|^{2}ds>N
\\ \mbox{or}\  &|{\bf u}_{n}(t-)|^{3}+\int^{t}_{0}|{\bf u}_{n}(s)|\|{\bf u}_{n}(s)\|^{2}ds>N\}.
\end{align*}
It follows from the It\^o formula that (see, e.g., \cite[Eq. (16), p. 65]{SPDE in hydrodynamic})
\begin{align}\label{Ito 3}
|{\bf u}_{n}(t)|^{3}&=|{\bf u}_{n}(0)|^{3}
\\&+3\int^{t}_{0}|{\bf u}_{n}(s)|\langle-\nu{\bf A}{\bf u}_{n}(s)-{\bf B}_{k_{\epsilon}}({\bf u}_{n}(s))+{\bf f}(s), {\bf u}_{n}(s),\rangle_{V}ds\nonumber
\\&+3\int^{t}_{0}|{\bf u}_{n}(s)|\langle{\bf u}_{n}(s), \sigma_{n}(s, {\bf u}_{n}(s),\mathfrak{r}(s))dW_{n}(s)\rangle\nonumber
\\&+3\int^{t}_{0}|{\bf u}_{n}(s)|\|\sigma_{n}(s, {\bf u}_{n}(s),\mathfrak{r}(s))\|^{2}_{L_{Q}}ds\nonumber
\\&+\int^{t}_{0}\int_{Z}\Big(|{\bf u}_{n}(s-)+{\bf G}_{n}(s, {\bf u}_{n}(s-),\mathfrak{r}(s), z)|^{3}\nonumber
\\&\qquad\qquad-|{\bf u}_{n}(s-)|^{3}\Big)\tilde{N}_{1}(dz, ds)\nonumber
\\&+\int\Big(|{\bf u}_{n}(s)+{\bf G}_{n}(s, {\bf u}_{n}(s), \mathfrak{r}(s), z)|^{3}-|{\bf u}_{n}(s)|^{3}\nonumber
\\&\qquad-3|{\bf u}_{n}(s)|\big({\bf u}_{n}(s), {\bf G}_{n}(s, {\bf u}_{n}(s), \mathfrak{r}(s), z)\big)_{H}\Big)\nu_{1}(dz)ds,\nonumber
\end{align}
where the last integral is over $[0, t]\times Z$.
%\begin{align}\label{Ito 3}
%|{\bf u}_{n}(t)|^{3}&=|{\bf u}_{n}(0)|^{3}
%\\&+3\int^{t}_{0}|{\bf u}_{n}(s)|\langle-\nu{\bf A}{\bf u}_{n}(s)-{\bf B}_{k_{\epsilon}}({\bf u}_{n}(s))+{\bf f}(s), {\bf u}_{n}(s),\rangle_{V}ds\nonumber
%\\&+3\int^{t}_{0}|{\bf u}_{n}(s)|\langle{\bf u}_{n}(s), \sigma_{n}(s, {\bf u}_{n}(s),\mathfrak{r}(s))dW_{n}(s)\rangle\nonumber
%\\&+3\int^{t}_{0}|{\bf u}_{n}(s)|\|\sigma_{n}(s, {\bf u}_{n}(s),\mathfrak{r}(s))\|^{2}_{L_{Q}}ds\nonumber
%\\&+\int^{t}_{0}\int_{Z}\Big(|{\bf u}_{n}(s-)+{\bf G}_{n}(s, {\bf u}_{n}(s-),%\mathfrak{r}(s), z)|^{3}\nonumber
%\\&\qquad\qquad-|{\bf u}_{n}(s-)|^{3}\Big)N_{1}(dz, ds)\nonumber
%\\&-3\int^{t}_{0}\int_{Z}|{\bf u}_{n}(s)|\big({\bf u}_{n}(s), {\bf G}_{n}(s, {\bf u}%_{n}(s),\mathfrak{r}(s), z)\big)_{H}\nu(dz)ds.\nonumber
%\end{align}
%Take integration up to $t\wedge\tau'_{N}$ and then expectation in \eqref{Ito 3}. 
%Then the martingale terms
%\begin{align*}
%3\int^{t}_{0}|{\bf u}_{n}(s)|\langle{\bf u}_{n}(s), \sigma_{n}(s, {\bf u}_{n}(s),\mathfrak{r}(s))dW_{n}(s)
%\end{align*}
%and
%\begin{align*}
%\int^{t}_{0}\int_{Z}\Big(|{\bf u}_{n}(s-)+{\bf G}_{n}(s, {\bf u}_{n}(s-),\mathfrak{r}(s), z)|^{3}-|{\bf u}_{n}(s-)|^{3}\Big)\tilde{N}_{1}(dz, ds)
%\end{align*}
%vanish. Hence, we obtain
Taking integration up to $t\wedge\tau'_{N}$ and then expectation in \eqref{Ito 3}, we obtain
\begin{align}
&\mathbb{E}|{\bf u}_{n}(t\wedge\tau'_{N})|^{3}+3\nu\mathbb{E}\int^{t\wedge\tau'_{N}}_{0}|{\bf u}_{n}(s)|\|{\bf u}_{n}(s)\|^{2}ds\label{a priori_1006}
\\&\leq \mathbb{E}|{\bf u}_{0}|^{3}+3\mathbb{E}\int^{t\wedge\tau'_{N}}_{0}|{\bf u}_{n}(s)|\|{\bf f}(s)\|_{V'}|\|{\bf u}_{n}(s)\|ds\nonumber
\\&+3\mathbb{E}\int^{t\wedge\tau'_{N}}_{0}|{\bf u}_{n}(s)|\|\sigma(s, {\bf u}_{n}(s), \mathfrak{r}(s))\|^{2}_{L_{Q}}ds\nonumber
\\&+\mathbb{E}\int^{t\wedge\tau'_{N}}_{0}\int_{Z}\Big(|{\bf u}_{n}(s)+{\bf G}_{n}(s, {\bf u}_{n}(s), \mathfrak{r}(s), z)|^{3}-|{\bf u}_{n}(s)|^{3}\nonumber
\\&\qquad\qquad\qquad\quad-3|{\bf u}_{n}(s)|\big({\bf u}_{n}(s), {\bf G}_{n}(s, {\bf u}_{n}(s),\mathfrak{r}(s), z)\big)_{H}\Big)\nu(dz)ds.\nonumber
\end{align}

An application of triangle inequality and   Hypothesis $\bf H3$ yields
\begin{align}
&\mathbb{E}\int^{t\wedge\tau'_{N}}_{0}\int_{Z}\Big(|{\bf u}_{n}(s)+{\bf G}_{n}(s, {\bf u}_{n}(s), \mathfrak{r}(s), z)|^{3}-|{\bf u}_{n}(s)|^{3}\label{a priori_1005}
\\&\qquad\qquad\qquad-3|{\bf u}_{n}(s)|\big({\bf u}_{n}(s), {\bf G}_{n}(s, {\bf u}_{n}(s),\mathfrak{r}(s), z)\big)_{H}\Big)\nu(dz)ds\nonumber
%\\&\leq\mathbb{E}\int^{t\wedge\tau'_{N}}_{0}\int_{Z}6|{\bf u}_{n}(s)|^{2}|{\bf G}_{n}(s, {\bf u}_{n}(s), \mathfrak{r}(s), z)|\nu(dz)ds\nonumber
%\\&\quad+\mathbb{E}\int^{t\wedge\tau'_{N}}_{0}\int_{Z}3|{\bf u}_{n}(s)||{\bf G}_{n}(s, {\bf u}_{n}(s), \mathfrak{r}(s), z)|^{2}\nu(dz)ds\nonumber
%\\&\quad+\mathbb{E}\int^{t\wedge\tau'_{N}}_{0}\int_{Z}|{\bf G}_{n}(s, {\bf u}_{n}(s), \mathfrak{r}(s), z)|^{3}\nu(dz)ds\nonumber
\\&\leq 10K\mathbb{E}\int^{t\wedge\tau'_{N}}_{0}|{\bf u}_{n}(s)|^{3}ds+6K\mathbb{E}\int^{t\wedge\tau'_{N}}_{0}|{\bf u}_{n}(s)|^{2}ds\nonumber
\\&\quad+3K\mathbb{E}\int^{t\wedge\tau'_{N}}_{0}|{\bf u}_{n}(s)|ds+KT.\nonumber
\end{align}

It follows from the basic Young inequality and the property $|\cdot|\leq\|\cdot\|$ that
\begin{align}\label{a priori_2001}
&3\|{\bf f}(s)\|_{V'}|{\bf u}_{n}(s)|\|{\bf u}_{n}(s)\|\leq\frac{1}{\nu^{2}}\|{\bf f}(s)\|^{3}_{V'}+2\nu(|{\bf u}_{n}(s)|\|{\bf u}_{n}(s)\|)^{\frac{3}{2}}\nonumber
\\&\leq\frac{1}{\nu^{2}}\|{\bf f}(s)\|^{3}_{V'}+2\nu|{\bf u}_{n}(s)|\|{\bf u}_{n}(s)\|^{2}.
\end{align}

Using  Hypothesis $\bf H1$, \eqref{a priori_1005}, and \eqref{a priori_2001} in \eqref{a priori_1006}, one has
\begin{align}
&\mathbb{E}|{\bf u}_{n}(t\wedge\tau'_{N})|^{3}+\nu\mathbb{E}\int^{t\wedge\tau'_{N}}_{0}|{\bf u}_{n}(s)|\|{\bf u}_{n}(s)\|^{2}ds\label{a priori_1007}
\\&\leq \mathbb{E}|{\bf u}_{0}|^{3}+\frac{1}{\nu^{2}}\mathbb{E}\int^{t\wedge\tau'_{N}}_{0}\|{\bf f}(s)\|^{3}_{V'}ds
+6K\mathbb{E}\int^{t\wedge\tau'_{N}}_{0}|{\bf u}_{n}(s)|^{2}ds\nonumber
\\&\quad+4K\mathbb{E}\int^{t\wedge\tau'_{N}}_{0}|{\bf u}_{n}(s)|ds+KT
+11k\mathbb{E}\int^{t\wedge\tau'_{N}}_{0}|{\bf u}_{n}(s)|^{3}ds.\nonumber
\end{align}
Notice that
$
\int^{t\wedge\tau'_{N}}_{0}|{\bf u}_{n}(s)|ds\leq\int^{t}_{0}|{\bf u}_{n}(s)|ds
$
since $t\wedge\tau'_{N}\leq t$. Thus, by the Schwarz inequality,  the Jensen inequality (for concave functions), the property that $|\cdot|\leq \|\cdot\|$, and \eqref{L^2 M}, we have
\begin{align}
&\mathbb{E}\int^{t\wedge\tau'_{N}}_{0}|{\bf u}_{n}(s)|ds\leq\mathbb{E}\int^{t}_{0}|{\bf u}_{n}(s)|ds\label{1-norm}
\\&\leq\mathbb{E}\Big(\int^{t}_{0}|{\bf u}_{n}(s)|^{2}ds\Big)^{\frac{1}{2}}\sqrt{T}\leq\sqrt{T}\Big(\mathbb{E}\int^{t}_{0}\|{\bf u}_{n}(s)\|^{2}ds\Big)^{\frac{1}{2}}\nonumber;
\end{align}
we also have 
\begin{align}
\mathbb{E}\int^{t\wedge\tau'_{N}}_{0}|{\bf u}_{n}(s)|^{2}ds\leq\mathbb{E}\int^{t}_{0}\|{\bf u}_{n}(s)\|^{2}ds\leq C.\label{2-norm}
\end{align}
Making use of \eqref{1-norm} and \eqref{2-norm} in \eqref{a priori_1007}, we then use the Gronwall inequality to obtain
\begin{align}\label{3-norm}
\mathbb{E}|{\bf u}_{n}(t\wedge\tau'_{N})|^{3}
\leq C(\mathbb{E}|{\bf u}_{0}|^{3}, \mathbb{E}\int^{T}_{0}\|{\bf f}(s)\|^{3}_{V'}ds, \nu, K, T).
\end{align}
Utilizing \eqref{3-norm} on the last term on the right of \eqref{a priori_1007}, we conclude
\begin{align*}
\mathbb{E}|{\bf u}_{n}(t\wedge\tau'_{N})|^{3}+\nu\mathbb{E}\int^{t\wedge\tau_{N}}_{0}|{\bf u}_{n}(s)|\|{\bf u}_{n}(s)\|^{2}ds%\label{a priori_1012}
\leq C(\mathbb{E}|{\bf u}_{0}|^{3}, \mathbb{E}\int^{T}_{0}\|{\bf f}(s)\|^{3}_{V'}ds, \nu, K, T).\nonumber
\end{align*}

A simplification of the last two terms in \eqref{Ito 3} gives
\begin{align*}
&\int^{t}_{0}\int_{Z}\Big(|{\bf u}_{n}(s-)+{\bf G}_{n}(s, {\bf u}_{n}(s-),\mathfrak{r}(s), z)|^{3}-|{\bf u}_{n}(s-)|^{3}\Big)N_{1}(dz, ds)
\\&+3\int^{t}_{0}\int_{Z}|{\bf u}_{n}(s)|\big|\big({\bf u}_{n}(s), {\bf G}_{n}(s, {\bf u}_{n}(s),\mathfrak{r}(s), z)\big)_{H}\big|\nu_{1}(dz)ds.
\end{align*}
Plugging the result above into \eqref{Ito 3}, we have 
\begin{align}\label{Ito 3'}
|{\bf u}_{n}(t)|^{3}&=|{\bf u}(0)|^{3}
\\&+3\int^{t}_{0}|{\bf u}_{n}(s)|\langle-\nu{\bf A}{\bf u}_{n}(s)-{\bf B}_{k}({\bf u}_{n}(s))+{\bf f}(s), {\bf u}_{n}(s),\rangle_{V}ds\nonumber
\\&+3\int^{t}_{0}|{\bf u}_{n}(s)|\langle{\bf u}_{n}(s), \sigma_{n}(s, {\bf u}_{n}(s),\mathfrak{r}(s))dW_{n}(s)\rangle\nonumber
\\&+3\int^{t}_{0}|{\bf u}_{n}(s)|\|\sigma_{n}(s, {\bf u}_{n}(s),\mathfrak{r}(s))\|^{2}_{L_{Q}}ds\nonumber
\\&+\int\Big(|{\bf u}_{n}(s-)+{\bf G}_{n}(s, {\bf u}_{n}(s-),\mathfrak{r}(s), z)|^{3}-|{\bf u}_{n}(s-)|^{3}\Big)N_{1}\nonumber
\\&-3\int^{t}_{0}\int_{Z}|{\bf u}_{n}(s)|\big({\bf u}_{n}(s), {\bf G}_{n}(s, {\bf u}_{n}(s),\mathfrak{r}(s), z)\big)_{H}\nu_{1}(dz)ds,\nonumber
\end{align}
where the second last integral is over $[0, t]\times Z$ and $N_{1}=N_{1}(dz, ds)$. Taking supremum over $T\wedge\tau'_{N}$ and then expectation on \eqref{Ito 3'}, we have
\begin{align}
&\mathbb{E}\sup_{0\leq t\leq T\wedge\tau'_{N}}|{\bf u}_{n}(t)|^{3}+3\nu\mathbb{E}\int^{T\wedge\tau'_{N}}_{0}|{\bf u}_{n}(s)|\|{\bf u}_{n}(s)\|^{2}ds\label{a priori_1008}
\\&\leq\mathbb{E}|{\bf u}(0)|^{3}+3\mathbb{E}\int^{T\wedge\tau'_{N}}_{0}|{\bf u}_{n}(s)|\langle{\bf f}(s), {\bf u}_{n}(s)\rangle_{V}ds
\nonumber
\\&\quad+\mathbb{E}\sup_{0\leq t\leq T\wedge\tau'_{N}}3\int^{t}_{0}|{\bf u}_{n}(s)|\langle{\bf u}_{n}(s),\sigma(s, {\bf u}_{n}(s), \mathfrak{r}(s))dW_{s}(s)\rangle\nonumber
\\&\quad+3\mathbb{E}\int^{T\wedge\tau'_{N}}_{0}|{\bf u}_{n}(s)|\|\sigma_{n}(s, {\bf u}_{n}(s),\mathfrak{r}(s))\|^{2}_{L_{Q}}ds\nonumber
\\&\quad+\mathbb{E}\int^{T\wedge\tau'_{N}}_{0}\int_{Z}\Big(|{\bf u}_{n}(s-)+{\bf G}_{n}(s, {\bf u}_{n}(s-),\mathfrak{r}(s), z)|^{3}-|{\bf u}_{n}(s-)|^{3}\Big)N_{1}(dz, ds)\nonumber
\\&\quad+3\mathbb{E}\int^{T\wedge\tau'_{N}}_{0}\int_{Z}|{\bf u}_{n}(s)|\big|\big({\bf u}_{n}(s), {\bf G}_{n}(s, {\bf u}_{n}(s),\mathfrak{r}(s), z)\big)_{H}\big|\nu(dz)ds.\nonumber
\end{align}
By the Davis inequality, we have
\begin{align}
&3\mathbb{E}\sup_{0\leq t\leq T\wedge\tau'_{N}}\int^{t}_{0}|{\bf u}_{n}(s)|\langle{\bf u}_{n}(s), \sigma(s, {\bf u}_{n}(s),\mathfrak{r}(s))dW_{n}(s)\rangle\nonumber %\label{a priori_1009}
\\&\leq3\sqrt{2}\mathbb{E}\Big\{\Big(\int^{T\wedge\tau'_{N}}_{0}\|\sigma(s, {\bf u}_{n}(s),\mathfrak{r}(s))(|{\bf u}_{n}(s)|{\bf u}_{n}(s))\|^{2}_{L_{Q}}\Big)^{\frac{1}{2}}\Big\}\nonumber
\\&\leq 3\sqrt{2}\mathbb{E}\Big\{\sup_{0\leq t\leq T\wedge\tau'_{N}}|{\bf u}_{n}(t)|^{2}\Big(\int^{T\wedge\tau'_{N}}_{0}\|\sigma(s, {\bf u}_{n}(s), \mathfrak{r}(s))\|^{2}_{L_{Q}}ds\Big)^{\frac{1}{2}}\Big\};\nonumber
\end{align}
invoking the basic Young inequality and Hypothesis $\bf H1$ and continuing, 
\begin{align*}
\leq 3\sqrt{2}\mathbb{E}\Big\{\frac{2}{3}\epsilon\sup_{0\leq t\leq T\wedge\tau'_{N}}|{\bf u}_{n}(t)|^{3}
+\frac{1}{3}C_{\epsilon}\Big(\int^{T\wedge\tau'_{N}}_{0}\|\sigma(s, {\bf u}_{n}(s), \mathfrak{r}(s))\|^{2}_{L_{Q}}ds\Big)^{\frac{3}{2}}\Big\};
\end{align*}
keep simplifying, one reaches
\begin{align*}
&\leq 3\sqrt{2}\mathbb{E}\Big\{\frac{2}{3}\epsilon\sup_{0\leq t\leq T\wedge\tau'_{N}}|{\bf u}_{n}(t)|^{3}
+\frac{1}{3}C_{\epsilon}\sqrt{T}\int^{T\wedge\tau'_{N}}_{0}\|\sigma(s, {\bf u}_{n}(s), \mathfrak{r}(s))\|^{3}_{L_{Q}}ds\Big\}\nonumber
\\&\leq 2\sqrt{2}\epsilon\mathbb{E}\sup_{0\leq t\leq T\wedge\tau'_{N}}|{\bf u}_{n}(s)|^{3}+\sqrt{2}\sqrt{T}K C_{\epsilon}\mathbb{E}\int^{T\wedge\tau'_{N}}_{0}|{\bf u}_{n}(s)|^{3}ds
+\sqrt{2}KT^{\frac{3}{2}}C_{\epsilon}.%\label{a priori_1009}
\end{align*}
In conclusion, one has
\begin{align}\label{a priori_1009}
&3\mathbb{E}\sup_{0\leq t\leq T\wedge\tau'_{N}}\int^{t}_{0}|{\bf u}_{n}(s)|\langle{\bf u}_{n}(s), \sigma(s, {\bf u}_{n}(s),\mathfrak{r}(s))dW_{n}(s)\rangle\nonumber
\\&\leq 2\sqrt{2}\epsilon\mathbb{E}\sup_{0\leq t\leq T\wedge\tau'_{N}}|{\bf u}_{n}(s)|^{3}+\sqrt{2}\sqrt{T}K C_{\epsilon}\mathbb{E}\int^{T\wedge\tau'_{N}}_{0}|{\bf u}_{n}(s)|^{3}ds
+\sqrt{2}KT^{\frac{3}{2}}C_{\epsilon}.
\end{align}
An application of triangle inequality and expanding the cubic power yields (with the integral on the left side being over $[0, T\wedge\tau'_{N}]$ and  $N_{1}=N_{1}(dz, ds)$.)
%and the Hypothesis $\bf H3$ leads
\begin{align}
&\mathbb{E}\int\Big(|{\bf u}_{n}(s-)+{\bf G}_{n}(s-, {\bf u}_{n}(s-),\mathfrak{r}(s-), z)|^{3}-|{\bf u}_{n}(s-)|^{3}\Big)N_{1}\nonumber %\label{1010}
\\&\leq 3\mathbb{E}\int^{T\wedge\tau'_{N}}_{0}\int_{Z}|{\bf u}_{n}(s-)|^{2}|{\bf G}_{n}(s-, {\bf u}_{n}(s-),\mathfrak{r}(s-), z)|N_{1}(dz, ds)\nonumber
\\&\quad+3\mathbb{E}\int^{T\wedge\tau'_{N}}_{0}\int_{Z}|{\bf u}_{n}(s-)||{\bf G}_{n}(s-, {\bf u}_{n}(s-),\mathfrak{r}(s-), z)|^{2}N_{1}(dz, ds)\nonumber
\\&\quad+\mathbb{E}\int^{T\wedge\tau'_{N}}_{0}\int_{Z}|{\bf G}_{n}(s, {\bf u}_{n}(s),\mathfrak{r}(s), z)|^{3}N_{1}(dz, ds);\nonumber
\end{align}
invoking Hypothesis $\bf H3$ and continuing, 
\begin{align}
%&\leq 3\mathbb{E}\int^{T\wedge\tau'_{N}}_{0}\int_{Z}|{\bf u}_{n}(s)|^{2}|{\bf G}_{n}(s, {\bf u}_{n}(s),\mathfrak{r}(s), z)|\nu(dz)ds\nonumber
%\\&\quad+3\mathbb{E}\int^{T\wedge\tau'_{N}}_{0}\int_{Z}|{\bf u}_{n}(s)||{\bf G}_{n}(s, {\bf u}_{n}(s),\mathfrak{r}(s), z)|^{2}\nu(dz)ds\nonumber
%\\&\quad+\mathbb{E}\int^{T\wedge\tau'_{N}}_{0}\int_{Z}|{\bf G}_{n}(s, {\bf u}%_{n}(s),\mathfrak{r}(s), z)|^{3}\nu(dz)ds\nonumber
&\leq7K\mathbb{E}\int^{T\wedge\tau'_{N}}_{0}|{\bf u}_{n}(s)|^{3}ds+3K(1+m)\mathbb{E}\int^{T\wedge\tau'_{N}}_{0}|{\bf u}_{n}(s)|^{2}ds\nonumber
\\&\quad+3K\mathbb{E}\int^{T\wedge\tau'_{N}}_{0}|{\bf u}_{n}(s)|ds+KT.\label{1010}
\end{align}
Employing \eqref{a priori_1009} and \eqref{1010} in \eqref{a priori_1008} and then using Hypotheses $\bf H1$ and $\bf H3$ and the basic Young inequality, we have
\begin{align*}
&\mathbb{E}\sup_{0\leq t\leq T\wedge\tau'_{N}}|{\bf u}_{n}(t)|^{3}+\nu\mathbb{E}\int^{T\wedge\tau'_{N}}_{0}|{\bf u}_{n}(s)|\|{\bf u}_{n}(s)\|^{2}ds
\\&\leq \mathbb{E}|{\bf u}(0)|^{3}+\frac{1}{\nu^{2}}\mathbb{E}\int^{T\wedge\tau'_{N}}_{0}\|{\bf f}(s)\|^{3}_{V'}ds
\\&\quad+2\sqrt{2}\epsilon\mathbb{E}\sup_{0\leq t\leq T\wedge\tau'_{N}}|{\bf u}_{n}(s)|^{3}+\sqrt{2}\sqrt{T}K C_{\epsilon}\mathbb{E}\int^{T\wedge\tau'_{N}}_{0}|{\bf u}_{n}(s)|^{3}ds
\\&\quad+\sqrt{2}KT^{\frac{3}{2}}C_{\epsilon}
\\&\quad+3K\mathbb{E}\int^{T\wedge\tau'_{N}}_{0}|{\bf u}_{n}(s)|ds+3K\mathbb{E}\int^{T\wedge\tau_{N}}_{0}|{\bf u}_{n}(s)|^{3}ds
\\&\quad+7K\mathbb{E}\int^{T\wedge\tau'_{N}}_{0}|{\bf u}_{n}(s)|^{3}ds+3K\mathbb{E}\int^{T\wedge\tau'_{N}}_{0}|{\bf u}_{n}(s)|^{2}ds\nonumber
\\&\quad+3K\mathbb{E}\int^{T\wedge\tau'_{N}}_{0}|{\bf u}_{n}(s)|ds+KT\nonumber
\\&\quad+3K\mathbb{E}\int^{T\wedge\tau'_{N}}_{0}|{\bf u}_{n}(s)|^{2}ds+3K\mathbb{E}\int^{T\wedge\tau'_{N}}_{0}|{\bf u}_{n}(s)|^{3}ds.
\end{align*}

Choose $\epsilon=\frac{1}{4\sqrt{2}}$. Then $C_{\epsilon}=4\sqrt{2}$. The inequality above can be simplified as
\begin{align*}
&\frac{1}{2}\mathbb{E}\sup_{0\leq t\leq T\wedge\tau'_{N}}|{\bf u}_{n}(t)|^{3}+\nu\mathbb{E}\int^{T\wedge\tau'_{N}}_{0}|{\bf u}_{n}(s)|\|{\bf u}_{n}(s)\|^{2}ds
\\&\leq C_{1}\mathbb{E}\int^{T\wedge\tau'_{N}}_{0}|{\bf u}_{n}(s)|^{3}ds+C_{2}\mathbb{E}\int^{T\wedge\tau'_{N}}_{0}|{\bf u}_{n}(s)|^{2}ds
\\&\quad+6K\mathbb{E}\int^{T\wedge\tau'_{N}}_{0}|{\bf u}_{n}(s)|ds+(8\sqrt{T}+1)KT,
\end{align*}
where $C_{1}=(13+8\sqrt{T})K$ and $C_{2}=6K$.
Using \eqref{1-norm}, \eqref{2-norm}, and \eqref{3-norm} in above, we conclude, upon a simplification, that
\begin{align*}
&\mathbb{E}\sup_{0\leq t\leq T\wedge\tau'_{N}}|{\bf u}_{n}(t)|^{3}+2\nu\mathbb{E}\int^{T\wedge\tau'_{N}}_{0}|{\bf u}_{n}(s)|\|{\bf u}_{n}(s)\|^{2}ds 
\\&\leq C(\mathbb{E}|{\bf u}(0)|^{3}, \mathbb{E}\int^{T}_{0}\|{\bf f}(s)\|^{3}_{V'}ds, \nu, K, T),
\end{align*}
which leads $T\wedge\tau'_{N}\rightarrow T$ as $N\rightarrow\infty$. Therefore, \eqref{L^3 sup M} is proved.
\end{proof}

\section{The Regularized Equation}\label{regularized}
The proof of Theorem \ref{existence and uniqueness of N-S M} consists of two parts: the existence of a weak solution (Theorem \ref{exist}) and the pathwise uniqueness of weak solutions (Theorem \ref{unique}). Once the weak solution is shown to be pathwise unique, then we apply a well-known result of Yamada and Watanabe \cite{Y-W} to deduce Theorem \ref{existence and uniqueness of N-S M}.

The argument is started by showing the existence of a weak solution; we remind the reader that the parameter $\epsilon$ that appears in ${\bf B}_{k_{\epsilon}}$ is chosen to be \emph{greater than} $0$ and \emph{fixed}.

\subsection{Existence of the solution to the regularized equation}\label{sec exist}
The existence of the weak solution is by studying the martingale problem posed by equation \eqref{N-S M}.  Suppose that ${\omega}^{\dagger}=({\bf u}, \mathfrak{r})$ is a solution to equation \eqref{N-S M}. Then it is not hard to see from the It\^o formula that
\begin{align}\label{mg_switching u}
M^{{\omega}^{\dagger}}(t)
:=F(t, {\bf u}(t), \mathfrak{r}(t))-F(0,{\bf u}(0), \mathfrak{r}(0))-\int^{t}_{0}\mathcal{L}F(s, {\bf u}(s), \mathfrak{r}(s))ds
\end{align}
is a $\mu$-martingale, where $\mu:=\mathcal{P}\circ{{\omega}^{\dagger}}^{-1}$ is the distribution of ${\omega}^{\dagger}$ and $\mathcal{L}$ is the operator introduced \eqref{operator L} with ${\bf B}_{k_{\epsilon}}$ in place of ${\bf B}$.

Recalling \eqref{the path space}, we have defined
$
\Omega^{\dagger}\colonequals\Omega^{\ast}\times\mathcal{D}([0, T]; \mathcal{S}).
$
 Now let $\omega=(u, i)$ be a generic element in $\Omega^{\dagger}$. Substituting $(\bf u, \mathfrak{r})$ by $(u, i)$ in \eqref{mg_switching u}, we obtain canonical expression:
\begin{align}\label{mg_switching}
M^{\omega}(t)
:=F(t, {u}(t), i(t))-F(0,{u}(0), {i}(0))-\int^{t}_{0}\mathcal{L}F(s, {u}(s), i(s))ds.
\end{align}
The aim of this subsection is to identify a measure $\mu$ in the path space $\Omega^{\dagger}$ under which $M(\cdot)$ in \eqref{mg_switching} is a martingale, and this is called \emph{the martingale problem posed by the stochastic Navier-Stokes equation with Markov switching} \eqref{N-S M}. 

Recalling the definition of path space $(\Omega^{\dagger}, \tau^{\dagger})$ from \eqref{the path space}, we let $\mathcal{B}$ denote the Borel $\sigma$-field of the topology $\tau^{\dagger}$. Define
$
\mathcal{F}_{t}:=\sigma(\omega(s): 0\leq s\leq t, \omega\in\Omega^{\dagger}).
$
Recall that $\mathcal{L}$ is the operator introduced in Section \ref{pre}, we are in a position to introduce the definition of a solution to a martingale problem.
\begin{definition}
A probability measure $\mu$ on $(\Omega, \mathcal{B})$ is called a solution of the martingale problem with the initial distribution $\mu_{0}$ and operator $\mathcal{L}$ is the following hold:
\begin{enumerate}
\item The time marginal of $\mu$ at $t=0$ is $\mu_{0}$, i.e., $\mu|_{t=0}=\mu_{0}$.
\item The canonical expression $M(t)$ defined in \eqref{mg_switching} is an $\mathcal{F}_{t}$-martingale.
\end{enumerate}
\end{definition}

Let $X_{t}(\omega)=\omega(t)$ for all $\omega\in\Omega^{\dagger}$ be the canonical process on $\Omega^{\dagger}$. Therefore, in terms of the canonical process, the definition becomes:
\begin{definition}\label{Solution of MG problem}
A process $X=\{{X}_{t}\}$ with path in $(\Omega^{\dagger}, \tau^{\dagger})$ defined on a probability space $(\Omega, \mathcal{F}, \mathcal{P})$ is called a solution to the martingale problem for the initial distribution $\mu_{0}$ and operator $\mathcal{L}$ if the following hold:
\begin{enumerate}
\item The distribution of $X_{0}$ is $\mu_{0}$.
\item For any $F\in\mathcal{D}(\mathcal{L})$, the process \eqref{mg_switching} is a $\mathcal{F}^{X}_{t}$-martingale.
\end{enumerate}
\end{definition}
There are several equivalent formulations of a solution to a martingale problem (see, e.g., \cite{Sundar MHD}), and we introduce one of them in the following lemma. The interested reader is referred to \cite[Prop. 7.1.2]{Sundar book} for more details.
\begin{lemma}\label{7.1.2}
The following statements are equivalent.
\begin{enumerate}
\item $X$ is a solution to the martingale problem for the operator ${\bf \mathcal{L}}$.
\item For all $f\in D({\bf\mathcal{L}})$, $0\leq t_{1}< t_{2}<\cdots<t_{n+1}$, $h_{1}, h_{2},\cdots h_{n}\in C_{b}$, and $n\geq 1$, we have
\begin{align*}
\mathbb{E}\Big\{\Big(f(X_{t_{n+1}})-f(X_{t_{n}})-\int^{t_{n+1}}_{t_{n}}{\bf\mathcal{L}}f(X_{s})ds\Big)\prod_{j=1}^{n}h_{j}(X_{t_{j}})\Big\}=0.
\end{align*}
\end{enumerate}
\end{lemma}

Let $\phi(t, i)$ be a real-valued bounded smooth function with compact support (in each variables). For $\rho\in\mathcal{D}({\bf A})\subseteq V$, $0\leq s\leq t$, and each generic element $\omega=(u, i)\in\Omega^{\dagger}$, define 
\begin{align}\label{M}
&M^{\phi}(t)-M^{\phi}(s)
\\&:=\phi(\langle {u}(t), \rho\rangle_{V}, i(t))-\phi(\langle {u}(s),\rho\rangle_{V}, i(s))\nonumber
\\&-\int^{t}_{s}\sum_{j=1}^{m}\gamma_{i(r-)j}\phi(\langle {u}(r),\rho\rangle_{V}, j)dr\nonumber
\\&-\int^{t}_{s}\Big(\phi^{'}(\langle {u}(r), \rho\rangle_{V}, i(r))\langle -\nu{\bf A}u(r)-{\bf B}_{k_{\epsilon}}({ u}(r))+{\bf f}(r),\rho\rangle_{V}\Big)dr\nonumber
\\&-\frac{1}{2}\int^{t}_{s}\phi^{''}(\langle {u}(r),\rho\rangle_{V}, i(r))\big(\rho, \sigma(r, {u}(r), i(r))Q\sigma^{*}(r, { u}(r), i(r))\rho\big)_{H}dr\nonumber
\\&-\int\Big(\phi(\langle{u}(r)+{\bf G}(r, {u}(r), i(r), z),\rho\rangle_{V}, i(r))-\phi(\langle{u}(r),\rho\rangle_{V}, i(r))\nonumber
\\&\qquad\quad-\phi^{'}(\langle{u}(r),\rho\rangle_{V}, i(r))\cdot\langle{\bf G}(r, {u}(r), i(r), z), \rho\rangle_{V}\Big)\nu(dz)dr,\nonumber
\end{align}
where the last integral is over $[s, t]\times Z$.

According to Lemma \ref{7.1.2}, to show $M^{\phi}(t)$ is a solution to the martingale problem, it suffices to find a Radon measure $\mu$ such that
\begin{align*}
\mathbb{E}^{\mu}\Big(\prod_{j=1}^{m}\psi_{j}(s_{j})(M^{\phi}(t)-M^{\phi}(s))\Big)=0, \ \forall s<s_{1}<\cdots<s_{m}<t, 
\end{align*}
where $\psi_{j}\in C_{b}(\Omega)$ and $\mathcal{F}_{s}$-measurable. 

Define the projection of $M^{\phi}(t)$ as follows:
\begin{align}\label{M_n}
&M^{\phi}_{n}(t)-M^{\phi}_{n}(s)
\\&:=\phi(\langle {u}(t), \rho\rangle_{V}, i(t))-\phi(\langle {u}(s),\rho\rangle_{V}, i(s))\nonumber
\\&-\int^{t}_{s}\sum_{j=1}^{m}\gamma_{i(r-)j}\phi(\langle {u}(r),\rho\rangle_{V}, j)dr\nonumber
\\&-\int^{t}_{s}\Big(\phi^{'}(\langle {u}(r), \rho\rangle_{V}, i(r))\cdot\langle -\nu{\bf A}_{n}u(r)-{\bf B}^{n}_{k_{\epsilon}}({u}(r))+{\bf f}(r),\rho\rangle_{V}\Big)dr\nonumber
\\&-\frac{1}{2}\int\phi^{''}(\langle {u}(r),\rho\rangle_{V}, i(r))\big(\rho, \sigma_{n}(r, {u}(r), i(r))Q\sigma^{*}_{n}(r, {u}(r), i(r))\rho\big)_{H}dr\nonumber
\\&-\int\Big(\phi(\langle{u}(r)+{\bf G}_{n}(r, { u}(r),i(r), z),\rho\rangle_{V}, i(r))-\phi(\langle{u}(r),\rho\rangle_{V}, i(r))\nonumber
\\&\qquad\quad-\phi^{'}(\langle{ u}(r),\rho\rangle_{V}, i(r))\langle{\bf G}_{n}(r, {u}(r),i(r) ,z), \rho\rangle_{V}\Big)\nu(dz)dr,\nonumber
\end{align}
where second last integral is over $[s, t]$, the last integral is over $[s, t]\times Z$, ${\bf A}_{n}=\Pi_{n}{\bf A}$, ${\bf B}^{n}_{k_{\epsilon}}=\Pi_{n}{\bf B}_{k_{\epsilon}}$, $\sigma_{n}=\Pi_{n}\sigma$, ${\bf G}_{n}=\Pi_{n}{\bf G}$, and $\omega=(u, i)\in\Omega^{\dagger}$ is a generic element.

Let $({\bf u}_{n}, \mathfrak{r})$ be the solution to equation \eqref{Galerkin M} and denote by $\mu_{n}$ the (joint) distribution of $({\bf u}_{n}, \mathfrak{r})$. Then it follows from the (finite dimensional) It\^o formula that $M^{\phi}_{n}(t)$ is a $\mu_{n}$-martingale, therefore, for all $n$,
$
\mathbb{E}^{\mu_{n}}\prod_{j=i}^{m}\psi_{j}(s_{j})M^{\phi}_{n}(t)=0, 
$
for all $s<s_{1}<\cdots<s_{m}<t$ and for $\psi_{j}\in C_{b}(\Omega)$ and $\mathcal{F}_{s}$-measurable. Hence,
\begin{align*}
\lim_{n\rightarrow\infty}\mathbb{E}^{\mu_{n}}\prod_{j=i}^{m}\psi_{j}(s_{j})M^{\phi}_{n}(t)=0, \ \forall s<s_{1}<\cdots<s_{m}<t, 
\end{align*}
for $\psi_{j}\in C_{b}(\Omega)$ and $\mathcal{F}_{s}$-measurable.

If we  show that

\begin{enumerate}
\item [$\bf M1.$] there exists a measure $\mu$ such that $\mu_{n}$  weakly converges to $\mu$,
\item [$\bf M2.$] $\lim_{n\rightarrow\infty}M^{\phi}_{n}(t)=M^{\phi}(t)$, and 
\item [$\bf M3.$] $\lim_{n\rightarrow\infty}\mathbb{E}^{\mu_{n}}M^{\phi}(t)=\mathbb{E}^{\mu}M^{\phi}(t)$,
\end{enumerate}
then it follows that $M^{\phi}(t)$ is a $\mu$-martingale.

Now we prove $\bf M1$. Recall that ${\bf u}_{n}$ is the solution to \eqref{Galerkin M} for each $n$.
\begin{lemma}\label{lemma 1}
The sequence $\{{\bf u}_{n}\}_{n=1}^{\infty}$ forms a relative compact set in the Skorohod space $\mathcal{D}([0, T]; V')$.
\end{lemma}

\begin{proof}
It is clear that $\{{\bf u}_{n}\}$ is a subset of $\mathcal{D}([0, T]; V')$.

Let $N>0$. By the Markov inequality, the property that $\|\cdot\|_{V'}\leq |\cdot|$, and \eqref{L^2 M}, we have
\begin{align*}
\mathcal{P}\Big(\|{\bf u}_{n}(t)\|_{V'}>N\Big)\leq\frac{1}{N^{2}}\mathbb{E}\|{\bf u}_{n}(t)\|^{2}_{V'}\leq\frac{1}{N^{2}}\mathbb{E}|{\bf u}_{n}(t)|^{2}\leq\frac{C}{N^{2}}.
\end{align*}
Therefore,
$
\lim_{N\rightarrow\infty}\limsup_{n}\mathcal{P}\Big(\|{\bf u}_{n}(t)\|_{V'}>N\Big)=0.
$

Let $(T_{n}, \delta_{n})$ be a sequence, where $T_{n}$ is a stopping time with $T_{n}+\delta_{n}\leq T$ and $\delta_{n}>0$ with $\delta_{n}\rightarrow 0$. For each $\epsilon>0$, the Chebyshev inequality implies 
\begin{align*}
&\mathcal{P}(\|{\bf u}_{n}(T_{n}+\delta_{n})-{\bf u}_{n}(T_{n})\|_{V'}>\epsilon)
\\&\leq\frac{1}{\epsilon^{2}}\mathbb{E}\|{\bf u}_{n}(T_{n}+\delta_{n})-{\bf u}_{n}(T_{n})\|^{2}_{V'}\leq\frac{1}{\epsilon^{2}}\mathbb{E}|{\bf u}_{n}(T_{n}+\delta_{n})-{\bf u}_{n}(T_{n})|^{2}.
\end{align*}
It follows from the It\^o formula and the Gronwall inequality that
\begin{align*}
\mathbb{E}|{\bf u}_{n}(T_{n}+\delta_{n})-{\bf u}_{n}(T_{n})|^{2}
\leq \Big(\frac{1}{\nu}\mathbb{E}\int^{\delta_{n}}_{0}\|{\bf f}(s)\|^{2}_{V'}ds+2K\delta_{n}\Big)e^{2K\delta_{n}},
\end{align*}
which tends to 0, as $n\rightarrow\infty$. Therefore, 
$\|{\bf u}_{n}(T_{n}+\delta_{n})-{\bf u}_{n}(T_{n})\|_{V'}\rightarrow 0$ 
in probability as $n\rightarrow\infty$. By Aldous' criterion, we conclude that $\{{\bf u}_{n}\}$ is tight in $\mathcal{D}([0, T]; V')$ and thus relative compact in it.
\end{proof}

We can have a even stronger convergence which is proved in the following proposition.

\begin{proposition}\label{cpt}
The sequence $\{{\bf u}_{n}\}_{n=1}^{\infty}$ forms a relative compact set in $L^{2}(0, T; H)$.
\end{proposition}

\begin{proof}
It follows from \eqref{L^2 sup M} that $\{{\bf u}_{n}\}$ is bounded in $L^{2}(0, T; V)$; also,  we have that
$
\mathbb{E}\int^{T}_{0}|{\bf u}_{n}(t)|^{2}dt\leq\mathbb{E}\int^{T}_{0}\|{\bf u}_{n}(t)\|^{2}dt\leq C,
$
which implies  
$\{{\bf u}_{n}\}\subset L^{2}(0, T; H)\cap \mathcal{D}([0, T]; V')$.
In addition, by Lemma \ref{lemma 1}, $\{{\bf u}_{n}\}$ is relatively compact in the space $\mathcal{D}([0, T]; V')$; the proposition follows from Lemma \ref{Metivier cpt}.
\end{proof}

Recalling Definition \ref{the path space of u} and ${\bf u}_{n}$ being the solution to \eqref{Galerkin M}, we deduce from a priori estimates and Banach-Alaoglu theorem that $\{{\bf u}_{n}\}$ is relatively compact in $(\Omega_{2}, \tau_{2})$ and $(\Omega_{3}, \tau_{3})$. In addition,  Lammta \ref{lemma 1} and \ref{cpt} imply that $\{{\bf u}_{n}\}$ is compact in $(\Omega_{1}, \tau_{1})$ and $(\Omega_{4},  \tau_{4})$, respectively. Therefore, by the Prohorov theorem, the induced distribution $\{\mu^{\ast}_{n}\}$\footnote{$\mu^{\ast}_{n}:=\mathcal{P}\circ{\bf u}_{n}^{-1}$} is tight on each space $(\Omega_{j}, \tau_{j})$ for $j=1, 2, 3, 4$. Hence, by (2) in Definition \ref{the path space of u}, $\{\mu^{\ast}_{n}\}$ is tight on $(\Omega^{\ast}, \tau)$. 

Let $\mu_{n}$ be the joint distribution of $({\bf u}_{n}, \mathfrak{r})$. Then $\{\mu_{n}\}_{n=1}^{\infty}$ is tight on the space $(\Omega^{\dagger}, \tau^{\dagger})$, hence, there exist a subsequence $\{\mu_{n_{\ell}}\}_{\ell=1}^{\infty}$ and a measure $\mu$ such that $\mu_{n_{\ell}}\Rightarrow \mu$. 

Next, we consider $\bf M2$. Recall from Section \ref{pre} that $H_{n}$ is the span of $\{e_{j}\}_{j=1}^{n}$ and $\Pi_{n}$ is a projection operator from $H$ onto $H_{n}$. Denote by $\{n_{\ell}\}_{\ell=1}^{\infty}$ the indices such that $\mu_{n_{\ell}}\Rightarrow\mu$.
\begin{lemma}\label{lemma rho}
For each $\rho\in D({\bf A})$, $\Pi_{n_{\ell}}\rho\rightarrow\rho$ in $V$, as $\ell\rightarrow\infty$.
\end{lemma}

\begin{proof}
Defining $f_{j}:=\frac{e_{j}}{\sqrt{\lambda_{j}}}$, one sees 
$
\|f_{j}\|=\frac{\|e_{j}\|^{2}}{\lambda_{j}}=\frac{\lambda_{j}}{\lambda_{j}}=1.
$
This implies that $\{f_{j}\}$ is a complete orthonormal basis in $V$. Thus,
\begin{align*}
\rho&=\sum_{j=1}^{\infty}(\rho, f_{j})_{V}f_{j};\quad
\Pi_{n_{\ell}}\rho:=\rho_{n_{\ell}}=\sum_{j=1}^{n_{\ell}}(\rho, f_{j})_{V}f_{j}.
\end{align*}
As a consequence,
$
\|\rho-\Pi_{n_{\ell}}\rho\|=\sum_{j=n_{\ell}+1}^{\infty}(\rho, f_{j})_{V}f_{j}\rightarrow 0,
$
as $\ell$ tends to infinity.
\end{proof}

\begin{lemma}\label{conv of A}
For each $\rho\in\mathcal{D}(\bf A)$, we have
\begin{align*}
\lim_{\ell\rightarrow\infty}\int^{t}_{s}\phi^{'}(\langle{u}(r), \rho\rangle_{V}, i(r))\langle-\nu{\bf A}_{n_{\ell}}{u}(r), \rho\rangle_{V}dr
=\int^{t}_{s}\phi^{'}(\langle{u}(r), \rho\rangle_{V}, i(r))\langle-\nu{\bf A}{u}(r), \rho\rangle_{V}dr,
\end{align*}
\end{lemma}

\begin{proof}
A direct computation gives, for almost all $r\in[t, s]$,
\begin{align*}
\langle-\nu{\bf A}_{n_{\ell}}{u}(r), \rho\rangle_{V}=-\nu\langle{\bf A}u(r), \rho_{n_{\ell}}\rangle\rightarrow-\nu\langle{\bf A}u(r), \rho\rangle
\end{align*}
as $\ell\rightarrow\infty$, by Lemma \ref{lemma rho}. In addition, 
\begin{align*}
|\langle-\nu{\bf A}_{n_{\ell}}{u}(r), \rho\rangle_{V}|=|-\nu\langle{\bf A}u(r), \rho_{n_{\ell}}\rangle|\leq\nu\|\rho\|_{V'}\|u(r)\|.
\end{align*}
Notice that $u\in\Omega^{\ast}$, therefore, $u\in L^{2}(0, T; V)\subset L(0, T; V)$. Hence, the lemma follows from the Lebesgue Dominated Convergence Theorem.
\end{proof}

\begin{lemma}\label{conv of B}
For each $\rho\in\mathcal{D}({\bf A})$, we have
\begin{align*}
\lim_{\ell\rightarrow\infty}\int^{t}_{s}\phi^{'}(\langle{ u}(r), \rho\rangle_{V}, i(r))\langle{\bf B}^{n_{\ell}}_{k_{\epsilon}}({u}(r)),\rho\rangle_{V}dr
=\int^{t}_{s}\phi^{'}(\langle{u}(r),\rho\rangle_{V}, i(r))\langle{\bf B}_{k_{\epsilon}}({u}(r)),\rho\rangle_{V}dr
\end{align*}
\end{lemma}

\begin{proof}
A similar argument as in Lemma \ref{conv of A} shows that
\begin{align*}
\langle{\bf B}^{n_{\ell}}_{k_{\epsilon}}({u}(r)),\rho\rangle_{V}\rightarrow\langle{\bf B}_{k_{\epsilon}}({u}(r)),\rho\rangle_{V}
\end{align*}
as $\ell\rightarrow\infty$ for almost all $r\in [s, t]$. In addition
\begin{align*}
|\langle{\bf B}^{n_{\ell}}_{k_{\epsilon}}({u}(r)),\rho\rangle_{V}|=|\langle{\bf B}_{k_{\epsilon}}(u(r)), \rho_{n_{\ell}}\rangle_{V}|\leq \|u(r)\||u(r)|\|\rho_{n_{\ell}}\|
\end{align*}
by \eqref{b_{k} uvu}. Since $u\in\Omega^{\ast}$, $u\in L^{2}(0, T; V)$. Thus, 
\begin{align*}
|\phi^{'}(\langle{u}(r),\rho\rangle_{V}, i)\langle{\bf B}_{k_{\epsilon}}({u}(r)),\rho\rangle_{V}|\leq\|\phi^{`}\|_{\infty}\|\rho\|\|u(r)\|^{2},
\end{align*}
which is an $L^{1}$-function. Therefore, the lemma follows from the Lebesgue Dominated Convergence Theorem.
\end{proof}

\begin{lemma}\label{conv of sigma}
For each $\rho\in\mathcal{D}({\bf A})$,  as $\ell\rightarrow\infty$,
\begin{align*}
\int^{t}_{s}\phi^{''}(\langle{u}(r),\rho\rangle_{V}, i(r))\big(\rho,\sigma_{n_{\ell}}(r,{u}(r),i(r) )Q\sigma^{*}_{n_{\ell}}(r, {u}(r),i(r) )\rho\big)_{H}dr
\end{align*}
converges to 
\begin{align*}
\int^{t}_{s}\phi^{''}(\langle{u}(r),\rho\rangle_{V}, i(r))\big(\rho,\sigma(r, {u}(r),i(r) )Q\sigma^{*}(r, {u}(r),i(r) )\rho\big)_{H}dr.
\end{align*}
\end{lemma}
\begin{proof}
Through out the proof, we write $\sigma(r)=\sigma(r, {u}(r),i(r) )$ and $(\cdot, \cdot)=\big(\cdot, \cdot\big)_{H}$. A direct computation shows that $
(\rho,\sigma_{n_{\ell}}(r)Q\sigma^{*}_{n_{\ell}}(r)\rho)$ equals to $(\sigma(r)Q\sigma^{\ast}(r)\rho_{n_{\ell}}, \rho_{n_{\ell}})$,
%\begin{align*}
%&\big(\rho,\sigma_{n_{\ell}}(r)Q\sigma^{*}_{n_{\ell}}(r)\rho\big)_{H}=\big(\rho_{n_{\ell}}, \sigma(r)Q\sigma^{\ast}_{n_{\ell}}(r)\rho\big)_{H}
%\\&=\big(\sigma_{n_{\ell}}(r)Q\sigma^{\ast}(r)\rho_{n_{\ell}}, \rho\big)_{H}=\big(\sigma(r)Q\sigma^{\ast}(r)\rho_{n_{\ell}}, \rho_{n_{\ell}}\big)_{H},
%\end{align*}
therefore,
\begin{align*}
(\rho,\sigma_{n_{\ell}}(r)Q\sigma^{*}_{n_{\ell}}(r)\rho)-(\rho,\sigma(r)Q\sigma^{*}(r)\rho)
=(\rho_{n_{\ell}},\sigma(r)Q\sigma^{*}(r)\rho_{n_{\ell}})-\big(\rho,\sigma(r)Q\sigma^{*}(r)\rho),
\end{align*}
which implies
\begin{align*}
&|(\rho,\sigma_{n_{\ell}}(r)Q\sigma^{*}_{n_{\ell}}(r)\rho)-\big(\rho,\sigma(r)Q\sigma^{*}(r)\rho)|
\\&\leq|(\rho_{n_{\ell}}, \sigma(r)Q\sigma^{\ast}(r)\rho_{n_{\ell}}-\sigma(r)Q\sigma^{\ast}(r)\rho)|
+|(\rho_{n_{\ell}}-\rho, \sigma(r)Q\sigma^{\ast}(r)\rho)|
\\&\leq |\rho_{n_{\ell}}| |\sigma(r)Q\sigma^{\ast}(r)| |\rho_{n_{\ell}}-\rho|
+|\rho_{n_{\ell}}-\rho| |\sigma(r)Q\sigma^{\ast}(r)||\rho|
\\&=2|\rho| |\rho_{n_{\ell}}-\rho|\|\sigma(r)\|_{L_{Q}}\leq2\|\rho\|\cdot\|\rho_{n_{\ell}}-\rho\|\cdot\|\sigma(r)\|_{L_{Q}}
\end{align*}
Thus, by Lemma \ref{lemma rho}, 
$
(\rho,\sigma_{n_{\ell}}(r)Q\sigma^{*}_{n_{\ell}}(r)\rho)\rightarrow\big(\rho,\sigma(r)Q\sigma^{*}(r)\rho),
$
as $\ell$ approaches  infinity, for all $r\in [s, t]$. In addition, 
\begin{align*}
&|\phi^{''}(\langle{u}(r),\rho\rangle_{V}, i(r))(\rho,\sigma_{n_{\ell}}(r,{u}(r),i(r) )Q\sigma^{*}_{n_{\ell}}(r, {u}(r),i(r) )\rho)|
\\&\leq \|\phi^{''}\|_{\infty}|\rho_{n_{\ell}}|\|\sigma(r)\|_{L_{Q}}\leq\|\phi^{''}\|_{\infty}\cdot\|\rho\|\cdot\|\sigma(r, u(r), i(r))\|_{L_{Q}}.
\end{align*}
Consider
\begin{align*}
\int^{T}_{0}\|\sigma(r, u(r), i(r))\|_{L_{Q}}dr\leq\sqrt{T}\Big(\int^{T}_{0}\|\sigma(r, u(r), i(r))\|^{2}_{L_{Q}}dr\Big)^{\frac{1}{2}}.
\end{align*}
Recall that $u\in\Omega^{\ast}$, therefore, $u\in L^{2}(0, T; H)$; the Hypothesis $\bf H1$ implies
\begin{align*}
\int^{T}_{0}\|\sigma(r, u(r), i(r))\|^{2}_{L_{Q}}dr\leq\int^{T}_{0}K(1+|u(r)|^{2}+i^{2})dr<C
\end{align*}
for a constant $C$. Therefore, we conclude that the function
\begin{align*}
|\phi^{''}(\langle{u}(r),\rho\rangle_{V}, i(r))(\rho,\sigma_{n_{\ell}}(r,{u}(r),i(r) )Q\sigma^{*}_{n_{\ell}}(r, {u}(r),i(r) )\rho)|
\end{align*}
is bounded by an $L^{1}$-function, hence, the lemma follows from the Lebesgue Dominated Convergence Theorem.
\end{proof}
\begin{lemma}\label{conv of G}
For each $\rho\in\mathcal{D}({\bf A})$, as $\ell\rightarrow\infty$
\begin{align*}
&\int^{t}_{s}\int_{Z}\Big(\phi(\langle{u}(r)+{\bf G}_{n_{\ell}}(r, {u}(r),i(r) , z),\rho\rangle_{V}, i(r))-\phi(\langle{u}(r),\rho\rangle_{V}, i(r))
\\&\qquad\qquad-\phi^{'}(\langle{u}(r),\rho\rangle_{V}, i(r))\cdot\langle{\bf G}_{n_{\ell}}(r, {u}(r),i(r) , z),\rho\rangle_{V}\Big)\nu_{1}(dz)dr
\end{align*}
converges to
\begin{align*}
&\int^{t}_{s}\int_{Z}\Big(\phi(\langle{u}(r)+{\bf G}(r, {u}(r), i(r) , z),\rho\rangle_{V}, i)-\phi(\langle{u}(r),\rho\rangle_{V}, i(r))
\\ &\qquad\qquad-\phi^{'}(\langle{u}(r),\rho\rangle_{V}, i(r))\cdot\langle{\bf G}(r, {u}(r), i(r), z),\rho\rangle_{V}\Big)\nu_{1}(dz)dr.
\end{align*}
\end{lemma}
\begin{proof}
Through out this proof, we write ${\bf G}(r, z)={\bf G}(r, {u}(r),i(r) , z)$, $\langle\cdot, \cdot\rangle=\langle\cdot, \cdot\rangle_{V}, u=u(r)$, and $i=i(r)$.

It follows from Lemma \ref{lemma rho} that
$
\langle{\bf G}_{n_{\ell}}(r, z),\rho\rangle=\langle{\bf G}(r , z), \rho_{n_{\ell}}\rangle,
$
%\begin{align*}
%\langle{\bf G}_{n_{\ell}}(r, {u}(r),i(r) , z),\rho\rangle_{V}=\langle{\bf G}(r, u(r), i(r) , z), \rho_{n_{\ell}}\rangle_{V},
%\end{align*}
which converges to
$
\langle{\bf G}(r, z), \rho\rangle,
$
as $\ell\rightarrow\infty$, for all $r\in[s, t]$. Therefore, the convergence of the integrand is shown. Next we argue that the integrand is bounded by an $L^{1}$-function, and thus the lemma follows from the the Lebesgue Dominated Convergence Theorem.

For a fixed $\ell$, writing
$a=\langle{u},\rho\rangle$ and
$b=\langle{u}+{\bf G}_{n_{\ell}}(r, z),\rho\rangle, $
we deduce from the mean value theorem that
$
\phi(\langle{u}+{\bf G}_{n_{\ell}}(r, z),\rho\rangle, i)-\phi(\langle{u}, \rho\rangle, i)
=\phi'(c, i)\langle{\bf G}_{n_{\ell}}(r, z),\rho\rangle,
$
where $c\in (a, b)$. Therefore, 
\begin{align*}
|\phi(\langle{u}+{\bf G}_{n_{\ell}}(r, z),\rho\rangle, i)-\phi(\langle{u},\rho\rangle, i)|
\leq \|\phi^{'}\|_{\infty}|\rho_{n_{\ell}}| |{\bf G}(r, z)|,
\end{align*}
which implies
\begin{align*}
&\Big|\Big(\phi(\langle{u}+{\bf G}_{n_{\ell}}(r , z),\rho\rangle, i)-\phi(\langle{u},\rho\rangle, i)
-\phi^{'}(\langle{u},\rho\rangle, i)\cdot\langle{\bf G}_{n_{\ell}}(r, z),\rho\rangle\Big)\Big|
\\&\leq |\phi(\langle{u}+{\bf G}_{n_{\ell}}(r, z),\rho\rangle, i)-\phi(\langle{u},\rho\rangle, i)|+|\phi^{'}(\langle{u},\rho\rangle, i)\langle{\bf G}_{n_{\ell}}(r, z),\rho\rangle|
\\&\leq 2\|\phi^{'}\|_{\infty}\|\rho\||{\bf G}(r, z)|.
\end{align*}
By Hypothesis $\bf H3$, $\int_{Z}|{\bf G}(r, z)|\nu(dz)$ is an $L^{1}$-function, therefore, the proof is complete.
\end{proof}

In light of Lemmata \ref{lemma rho} to \ref{conv of G}, $\bf M2$ has been proved. Moreover, as shown in the proofs of Lemmata \ref{lemma rho} to \ref{conv of G}, the expectation and limit is exchangeable, i.e., 
$
\lim_{n\rightarrow\infty}\mathbb{E}M^{\phi}_{n}(t)=\mathbb{E}\lim_{n\rightarrow\infty}M^{\phi}_{n}(t).
$

Lastly, we consider $\bf M3$. Clearly, if the assumption of Lemma \ref{prop conv} is fulfilled, then $\bf M3$ is obtained. So far, we have a sequence of measures $\{\mu_{n}\}_{n=1}^{\infty}$, and there exists a measure $\mu$ such that $\mu_{n_{\ell}}\Rightarrow\mu$ as $\ell\rightarrow\infty$. Therefore, it remains to prove that
\begin{enumerate}
\item $M^{\phi}(t)$ is continuous on $(\Omega^{\dagger}, \tau^{\dagger})$, and
\item  for some $\delta>0$, $\sup_{n}\mathbb{E}^{\mu_{n}}\Big(|M^{\phi}(t)|^{1+\delta}\Big)<C$, where $C$ is a constant.
\end{enumerate}

We begin the proof of continuity of $M^{\phi}(t)$ with the following auxiliary lemma.
\begin{lemma}\label{lemma phi}
Let $\{u_{n}\}_{n=1}^{\infty}$ and $u$ be members of $(\Omega^{\ast}, \tau)$ with $u_{n}\rightarrow u$ as $n\rightarrow\infty$ in $\tau$-topology. For almost all $t\in[0, T]$, $k=0, 1, 2$, and each $i\in\mathcal{S}$, we have 
\begin{align*} 
\frac{d^{k}}{dx^{k}}\phi(\langle {u}_{n}(t), \rho\rangle_{V}, i)\rightarrow\frac{d^{k}}{dx^{k}}\phi(\langle{ u}(t),\rho\rangle_{V}, i),
\end{align*}
as $n\rightarrow\infty$.
\end{lemma}
\begin{proof}
Denote
$
C({u}):=\{t\in[0, T]; \mathcal{P}\big({u}(t)={u}(t-)\big)=1\}.
$
It is known that the complement of $C({u})$ is at most countable (see, e.g., \cite{Billingsley}). Therefore, for almost all $t\in[0, T]$, one has
$
{u}_{n}(t)\rightarrow{u}(t), 
$
as $n\rightarrow\infty$. This further implies that
$
\langle{u}_{n}(t),\rho\rangle_{V}\rightarrow\langle{u}(t), \rho \rangle_{V}
$
, as $n\rightarrow\infty$ for any $\rho\in\mathcal{D}({\bf A})$. Therefore, the lemma follows from the smoothness of the function $\phi$.
\end{proof}

\begin{lemma}\label{lemma conti}
$M^{\phi}(t)$ is continuous in the $(\tau^{\dagger})$-topology.
\end{lemma}

\begin{proof}
It suffices to prove that $M^{\phi}(t)$ is continuous in the $\tau$-topology since there is no convergence issue in $\mathfrak{r}$.

Let $\{u_{n}\}$ and $u$ be members of $(\Omega^{\ast}, \tau)$ with $u_{n}\rightarrow u$ as $n\rightarrow\infty$ in $\tau$-topology. Let $M^{\phi}(u_{n}(t))$ be the function where $u_{n}$ is in place of $u$ in \eqref{M}. Given $u_{n}\rightarrow u$. We need to show that $\lim_{n\rightarrow\infty}M^{\phi}(u_{n}(t))=M^{\phi}(u(t))$, and we prove it by taking the term-by-term limit.

The first three terms follows from Lemma \ref{lemma phi} and the Bounded Convergence Theorem.

From now on, we write $u=u(r), u_{n}=u_{n}(r), i=i(r), (\cdot, \cdot)=\big(\cdot, \cdot\big)_{H}$, and $\langle\cdot, \cdot\rangle=\langle\cdot, \cdot\rangle_{V}$.

For the $\bf A$ term, it is not hard to see that 
$\langle{\bf A}u_{n}, \rho\rangle$ equals to $\langle u_{n}, {\bf A}\rho\rangle$ which converges to $\langle u, {\bf A}\rho\rangle$,
%\begin{align*}
%\langle{\bf A}u_{n}, \rho\rangle=\langle u_{n}, {\bf A}\rho\rangle\rightarrow\langle %u(r), {\bf A}\rho\rangle_{V}
%\end{align*}
as $n\rightarrow\infty$, for almost all $r\in [s, t]$ by Lemma \ref{lemma phi}. Consider
$
|\langle{\bf A}u_{n}, \rho\rangle|^{2}\leq\|u_{n}\|^{2}\|\rho\|^{2}_{V'},
$
and
$
\int^{T}_{0}\|u_{n}\|^{2}\|\rho\|^{2}_{V'}dr
\leq\|\rho\|^{2}_{V'}\int^{T}_{0}\|u_{n}\|^{2}dr<C
$
for all $n$ and a constant $C$ since $u_{n}\rightarrow u$ in $\tau$-topology and thus in $\tau_{2}$. This implies that
$
\sup_{n}\int^{T}_{0}|\langle{\bf A}u_{n}, \rho\rangle|^{2}dr<\infty
$
so that $\{\langle{\bf A}u_{n}, \rho\rangle\}$ is uniformly integrable, and hence
\begin{align*}
\lim_{n\rightarrow\infty}\int^{t}_{s}\langle-\nu{\bf A}u_{n}, \rho\rangle dr=\int^{t}_{s}\langle-\nu{\bf A}u, \rho\rangle dr.
\end{align*}

For the ${\bf B}_{k_{\epsilon}}$ term, using the definition of ${\bf B}_{k_{\epsilon}}$, we have
\begin{align*}
|\langle{\bf B}_{k_{\epsilon}}(u_{n}), \rho\rangle-\langle{\bf B}_{k_{\epsilon}}(u, \rho\rangle|
\leq |b(k_{\epsilon}u_{n}, u_{n}-u, \rho)|+|b(k_{\epsilon}(u_{n}-u), u, \rho)|,
\end{align*}
which together with \eqref{b_{k}uvw} further imply
\begin{align*}
|\langle{\bf B}_{k_{\epsilon}}(u_{n}), \rho\rangle-\langle{\bf B}_{k_{\epsilon}}(u), \rho\rangle|
\leq2\|\rho\|\|u_{n}\|\|u_{n}-u\|^{\frac{1}{2}}|u_{n}-u|^{\frac{1}{2}}.
\end{align*}
Therefore, the Schwarz inequality implies (with $C=2\|\rho\|$)
\begin{align*}
&\int^{T}_{0}|\langle{\bf B}_{k_{\epsilon}}(u_{n}), \rho\rangle-\langle{\bf B}_{k_{\epsilon}}(u), \rho\rangle|dr
\\&\leq C\|\Big(\int^{T}_{0}\|u_{n}\|^{2}dr\Big)^{\frac{1}{2}}\Big(\int^{T}_{0}\|u_{n}(r)-u(r)\||u_{n}-u|dr\Big)^{\frac{1}{2}}
\\&\leq C\|\Big(\int^{T}_{0}\|u_{n}\|^{2}dr\Big)^{\frac{1}{2}}
\Big(\int^{T}_{0}\|u_{n}\||u_{n}-u|dr+\int^{T}_{0}\|u\||u_{n}-u|dr\Big)^{\frac{1}{2}}
\\&\leq C\|\Big(\int^{T}_{0}\|u_{n}\|^{2}dr\Big)^{\frac{1}{2}}
\\&\quad\cdot\Big\{\Big(\int^{T}_{0}\|u_{n}\|^{2}dr\Big)^{\frac{1}{2}}\Big(\int^{T}_{0}|u_{n}-u|^{2}dr\Big)^{\frac{1}{2}}
+\Big(\int^{T}_{0}\|u\|^{2}dr\Big)^{\frac{1}{2}}\Big(\int^{T}_{0}|u_{n}-u|^{2}dr\Big)^{\frac{1}{2}}\Big\}^{\frac{1}{2}}.
\end{align*}
Since $u_{n}$ and $u$ are members of $\Omega^{\ast}$, the $L^{2}(0, T; V)$-norms are finite. In addition, $u_{n}\rightarrow u$ in $\tau$-topology implies that $u_{n}\rightarrow u$ in $\tau_{4}$ (the strong topology in $L^{2}(0, T; H)$). Hence, we conclude that
\begin{align*}
\lim_{n\rightarrow\infty}\int^{T}_{0}|\langle{\bf B}_{k_{\epsilon}}(u_{n}), \rho\rangle-\langle{\bf B}_{k_{\epsilon}}(u), \rho\rangle|dr=0,
\end{align*}
which implies
$$
\lim_{n\rightarrow\infty}\int^{t}_{s}\langle{\bf B}_{k_{\epsilon}}(u_{n}), \rho\rangle dr=\int^{t}_{s}\langle{\bf B}_{k_{\epsilon}}(u), \rho\rangle dr.
$$

For the term represents the continuous noise, consider
\begin{align*}
&\Big|(\rho, \sigma(r, u_{n}, i)Q\sigma^{\ast}(r, u_{n}, i)\rho)-(\rho, \sigma(r, u, i)Q\sigma^{\ast}(r, u, i)\rho)\Big|
\\&\leq \Big|\rho\Big|\cdot\Big\{\Big|\sigma(r, u_{n}, i)Q\sigma^{\ast}(r, u_{n}, i)\rho
-\sigma(r, u, i)Q\sigma^{\ast}(r, u, i)\rho\Big|\Big\}
\\&\leq |\rho|^{2}\Big\{\Big| \sigma(r, u_{n}, i)Q\sigma^{\ast}(r, u_{n}, i)
-\sigma(r, u, i)Q\sigma^{\ast}(r, u, i)\Big|\Big\}.
\end{align*}
Recalling the definition of $L_{Q}$-norm, we see that
\begin{align*}
\Big| \sigma(r, u_{n}, i)Q\sigma^{\ast}(r, u_{n}, i)-\sigma(r, u, i)Q\sigma^{\ast}(r, u, i)\Big|
=\|\sigma(r, u_{n}, i)-\sigma(r, u, i)\|_{L_{Q}}.
\end{align*}
Therefore, by  Hypothesis $\bf H2$, we have
\begin{align*}
&\int^{T}_{0}\Big|(\rho, \sigma(r, u_{n}, i)Q\sigma^{\ast}(r, u_{n}, i)\rho)
-(\rho, \sigma(r, u, i)Q\sigma^{\ast}(r, u, i)\rho)\Big|^{2}dr
\\&\leq\|\rho\|^{4}\int^{T}_{0}\|\sigma(r, u_{n}, i)-\sigma(r, u, i)\|^{2}_{L_{Q}}dr
\leq L\|\rho\|^{4}\int^{T}_{0}|u_{n}-u|^{2}dr,
\end{align*}
which approaches  $0$ as $n\rightarrow\infty$ since $u_{n}\rightarrow u$ in $\tau$ means that $u_{n}\rightarrow u$ in $\tau_{4}$ (the strong topology in $L^{2}(0, T; H)$). Thus, we have
\begin{align*}
\lim_{n\rightarrow\infty}\int^{t}_{s}(\rho, \sigma(r, u_{n}, i)Q\sigma^{\ast}(r, u_{n}, i)\rho)
=\int^{t}_{s}(\rho, \sigma(r, u, i)Q\sigma^{\ast}(r, u, i)\rho).
\end{align*}

For the jump noise term, notice that $\bf G$ is continuous in all of its components, therefore, 
$
\lim_{n\rightarrow\infty}{\bf G}(r, u_{n}, i, z)={\bf G}(r, u, i, z)
$
for almost all $r\in [s, t]$ and all fixed $z$. This implies that, by Lemma \ref{lemma phi},
\begin{align*}
\phi(\langle{u}_{n}+{\bf G}(r, {u}_{n}, i , z),\rho\rangle, i)-\phi(\langle{u}_{n},\rho\rangle, i)
-\phi^{'}(\langle{u}_{n},\rho\rangle, i)\cdot\langle{\bf G}(r, {u}_{n}, i, z),\rho\rangle
\end{align*}
converges to 
\begin{align*}
\phi(\langle{u}+{\bf G}(r, {u}, i , z),\rho\rangle, i)-\phi(\langle{u},\rho\rangle, i)
-\phi^{'}(\langle{u},\rho\rangle, i)\cdot\langle{\bf G}(r, {u}, i , z),\rho\rangle
\end{align*}
almost surely in $[s, t]\times Z$. Writing $a_{n}=\langle u_{n}, \rho\rangle$ and $b_{n}=\langle u_{n}+{\bf G}(r, u_{n}, i, z)\rangle$, one infers from the Mean Value Theorem that
\begin{align*}
\phi(\langle{u}_{n}+{\bf G}(r, {u}_{n}, i, z),\rho\rangle, i)-\phi(\langle{u}_{n},\rho\rangle, i)
=\phi^{'}(c_{n})\langle{\bf G}(r, u_{n}, i, z), \rho\rangle,
\end{align*}
where $\phi^{'}(c_{n})=\phi^{'}(c_{n}, i)$ and $c_{n}\in (a_{n}, b_{n})$. Therefore, 
\begin{align*}
&\Big|\phi(\langle{u}_{n}+{\bf G}(r, {u}_{n}, i, z),\rho\rangle, i)-\phi(\langle{u}_{n},\rho\rangle, i)
-\phi^{'}(\langle{u}_{n},\rho\rangle, i)\cdot\langle{\bf G}(r, {u}_{n}, i, z),\rho\rangle\Big|
\\&\leq\Big|\phi(\langle{u}_{n}+{\bf G}(r, {u}_{n}, i, z),\rho\rangle, i)-\phi(\langle{u}_{n},\rho\rangle, i)\Big|
+\Big|\phi^{'}(\langle{u}_{n},\rho\rangle, i)\cdot\langle{\bf G}(r, {u}_{n}, i , z),\rho\rangle\Big|
\\&=\Big|\phi^{'}(c_{n})\langle{\bf G}(r, u_{n}, i, z), \rho\rangle\Big|
+\Big|\phi^{'}(\langle{u}_{n},\rho\rangle, i)\cdot\langle{\bf G}(r, {u}_{n}, i, z),\rho\rangle\Big|
\leq2\|\phi^{'}\|_{\infty}|\rho||{\bf G}(r, u_{n}, i, z)|,
\end{align*}
which implies
\begin{align*}
&\int^{T}_{0}\int_{Z}\Big|\phi(\langle{u}_{n}+{\bf G}(r, {u}_{n}, i, z),\rho\rangle, i)-\phi(\langle{u}_{n},\rho\rangle, i)
\\&\qquad\qquad-\phi^{'}(\langle{u}_{n},\rho\rangle, i)\cdot\langle{\bf G}(r, {u}_{n}, i, z),\rho\rangle\Big|^{2}\nu_{1}(dz)dr
\\&\leq4\|\rho\|^{2}\int^{T}_{0}\int_{Z}|{\bf G}(r, u_{n}, i, z)|^{2}\nu_{1}(dz)dr
\leq C\int^{T}_{0}(1+|u_{n}|^{2})dr,
\end{align*}
where $C=4K\|\rho\|^{2}$ and the last inequality follows from  Hypothesis $\bf H3$. Since $u_{n}\rightarrow u$ in $\tau$-topology, $u_{n}\rightarrow u$ in $\tau_{4}$, which implies that
$
\sup_{n}\int^{T}_{0}|u_{n}|^{2}dr<C
$
for a constant $C$. Therefore, 
\begin{align*}
&\sup_{n}\int^{T}_{0}\int_{Z}\Big|\phi(\langle{u}_{n}+{\bf G}(r, {u}_{n}, i , z),\rho\rangle, i)-\phi(\langle{u}_{n},\rho\rangle, i)
\\&\qquad\qquad\qquad-\phi^{'}(\langle{u}_{n},\rho\rangle, i)\cdot\langle{\bf G}(r, {u}_{n}, i, z),\rho\rangle\Big|^{2}\nu_{1}(dz)dr<C.
\end{align*} 
 Hence, we conclude, as $n\rightarrow\infty$,
\begin{align*}
&\int^{T}_{0}\int_{Z}\Big(\phi(\langle{u}_{n}+{\bf G}(r, {u}_{n}, i, z),\rho\rangle, i)-\phi(\langle{u}_{n},\rho\rangle, i)
\\&\qquad\qquad-\phi^{'}(\langle{u}_{n},\rho\rangle, i)\cdot\langle{\bf G}(r, {u}_{n}, i, z),\rho\rangle\Big)\nu_{1}(dz)dr
\end{align*}
converges to 
\begin{align*}
&\int^{t}_{s}\int_{Z}\Big(\phi(\langle{u}+{\bf G}(r, {u}, i, z),\rho\rangle, i)-\phi(\langle{u}(r),\rho\rangle, i)
\\&\qquad\qquad-\phi^{'}(\langle{u},\rho\rangle, i)\cdot\langle{\bf G}(r, {u}, i, z),\rho\rangle\Big)\nu_{1}(dz)dr.
\end{align*}
Herein, the proof is complete.
\end{proof}

The following lemma the final piece of the required argument. It is the only place where we require $\mathbb{E}|{\bf u}_{0}|^{3}<0$ and ${\bf f}\in L^{3}(0, T; V')$. 
\begin{lemma}\label{lemma 2}
Suppose that the Hypotheses $\bf H$ is fulfilled, $\mathbb{E}|{\bf u}_{0}|^{3}<\infty$, and ${\bf f}\in L^{3}(0, T; V')$.
There exist some $\delta>0$ such that
$$
\sup_{\ell}\mathbb{E}^{\mu_{n_{\ell}}}\big[|M^{\phi}|^{1+\delta}\big]\leq C,
$$
where $C$ is an appropriate constant.
\end{lemma}
\begin{proof}
Recalling from \eqref{M} the definition of $M^{\phi}$, we employ inequality $(\sum_{i=1}^{5}a_{i})^{p}\leq 5^{p-1}\sum_{i=1}^{5} a_{i}^{p}$ and the Mean Value Theorem (on {\bf G}) to deduce that $|M^{\phi}(t)|^{1+\delta}$ is less than or equal to
\begin{align*}
&5^{\delta}|\phi(\langle{u}(t),\rho\rangle_{V}, i(t))|^{1+\delta}+5^{\delta}|\phi(\langle{u}(s)\rangle_{V}, i(s))|^{1+\delta}
\\&+5^{\delta}\|\phi'\|_{\infty}\int^{t}_{s}|\langle\nu{\bf A}{u}(r)+{\bf B}_{k_{\epsilon}}({u}(r))+{\bf f}(r),\rho\rangle_{V}|^{1+\delta}dr
\\&+5^{\delta}\|\phi''\|_{\infty}\int^{t}_{s}|(\rho, \sigma(r, {u}(r), i(r))Q\sigma^{\ast}(r, {u}(r), i(r)))_{H}|^{1+\delta}dr
\\&+5^{\delta}\int^{t}_{s}\|\phi'\|_{\infty}\Big|\int_{Z}\langle{\bf G}(r, {u}(r), i(r), z),\rho\rangle_{V}\nu(dz)\Big|^{1+\delta}dr
\end{align*}
since $\phi$ is bounded smooth function.

For the $\bf A$ term,
\begin{align*}
\int^{t}_{s}|\nu\langle{\bf A}{u}(r),\rho\rangle_{V}|^{1+\delta}dr
\leq\nu^{1+\delta}\int^{t}_{s}\big(\|{u}(r)\|_{V}\|{\bf A}\rho\|_{V'}\big)^{1+\delta}dr\leq C_{1}\int^{t}_{s}\|{u}(r)\|^{1+\delta}_{V}dr,
\end{align*}
where $C_{1}=\nu^{1+\delta}\|{\bf A}\rho\|_{V'}$.
This implies that
\begin{align*}
\mathbb{E}^{\mu_{n_{\ell}}}\int^{t}_{s}|\nu\langle{\bf A}{u}(r),\rho\rangle_{V}|^{1+\delta}dr
\leq C_{1}\mathbb{E}^{\mu_{n_{\ell}}}\int^{t}_{s}\|{u}(r)\|^{1+\delta}_{V}dr=C_{1}\mathbb{E}\int^{t}_{s}\|{\bf u}_{n_{\ell}}(r)\|^{1+\delta}_{V}dr,
\end{align*}
hence,
$
\sup_{\ell}\mathbb{E}^{n_{\ell}}\int^{t}_{s}\big|\langle\nu{\bf A}{u}(r), \rho\rangle_{V}\big|^{1+\delta}dr\leq C_{\bf{A}}
$
if $\delta<1$.

For the nonlinear term, \eqref{b_{k}uvw} and H\"older inequality imply
\begin{align*}
&\mathbb{E}^{\mu_{n_{\ell}}}\Big\{\big|\int^{t}_{s}\langle{\bf B}_{k_{\epsilon}}({u}(r)),\rho\rangle_{V}dr\big|^{1+\delta}\Big\}
\leq \|\rho\|^{1+\delta}_{V}\mathbb{E}^{\mu_{n_{\ell}}}\Big\{\big(\int^{t}_{s}\|{u}(r)\|_{V}|{u}(r)|_{H}dr\big)^{1+\delta}\Big\}
\\&\leq \|\rho\|^{1+\delta}_{V}\Big\{\mathbb{E}^{\mu_{n_{\ell}}}\big(\sup_{0\leq t\leq T}|{u}(t)|^{1+\delta}_{H}\big)^{p}\Big\}^{\frac{1}{p}}\Big\{\mathbb{E}^{\mu_{n_{\ell}}}\big(\int^{t}_{s}\|{u}(r)\|^{1+\delta}_{V}dr\big)^{q}\Big\}^{\frac{1}{q}},
\end{align*}
where $\frac{1}{p}+\frac{1}{q}=1$. Choosing $q$ such that $(1+\delta)q=2$,  we have
\begin{align}
&\mathbb{E}^{\mu_{n_{\ell}}}\Big\{\big|\int^{t}_{s}\langle{\bf B}_{k_{\epsilon}}({u}(r)),\rho\rangle_{V}\big|^{1+\delta}\Big\}\label{mg s_20}
\\&\leq \|\rho\|^{1+\delta}_{V}\Big\{\mathbb{E}^{\mu_{n_{\ell}}}\sup_{0\leq t\leq T}|{u}(t)|^{2(\frac{1+\delta}{1-\delta})}_{H}\Big\}^{\frac{1-\delta}{2}}\Big\{\mathbb{E}^{\mu_{n_{\ell}}}\int^{t}_{s}\|{u}(r)\|^{2}_{V}dr\Big\}^{\frac{1+\delta}{2}}.\nonumber
\\&= \|\rho\|^{1+\delta}_{V}\Big\{\mathbb{E}\sup_{0\leq t\leq T}|{\bf u}_{n_{\ell}}(t)|^{2(\frac{1+\delta}{1-\delta})}_{H}\Big\}^{\frac{1-\delta}{2}}\Big\{\mathbb{E}\int^{t}_{s}\|{\bf u}_{n_{\ell}}(r)\|^{2}_{V}dr\Big\}^{\frac{1+\delta}{2}}\nonumber
\end{align}
Taking $\delta=\frac{1}{5}$, we have $2(\frac{1+\delta}{1-\delta})=3$. The first expectation on the right of \eqref{mg s_20} will have a uniform bound by \eqref{L^3 sup M} if we further assume that $\mathbb{E}|{\bf u}_{0}|^{3}<\infty$. The boundedness of $\mathbb{E}\int^{t}_{s}\|{\bf u}_{n_{\ell}}(r)\|^{2}_{V}dr$ is followed from \eqref{L^2 M}. Therefore, \eqref{mg s_20} implies that, if $\delta\leq\frac{1}{5}$,
$$
\sup_{\ell}\mathbb{E}^{\mu_{n_{\ell}}}\Big\{\big|\int^{t}_{s}\langle{\bf B}_{k_{\epsilon}}({u}(r)),\rho\rangle_{V}\big|^{1+\delta}dr\Big\}\leq C_{{\bf B}}.
$$

For martingale terms, we have
\begin{align*}
&\int^{t}_{s}\big|\big(\rho, \sigma(r, {u}(r), i(r))Q\sigma^{\ast}(r, {u}(r), i(r))\rho\big)_{H}\big|^{1+\delta}dr
\\&\leq |\rho|^{2(1+\delta)}_{H}\int^{t}_{s}\|\sigma(r, {u}(r), i(r))\|^{1+\delta}_{L_{Q}}dr
\leq |\rho|^{2(1+\delta)}_{H}K^{\frac{1+\delta}{2}}\int^{t}_{s}(1+|{u}(r)|^{2}_{H})^{\frac{1+\delta}{2}}dr
\\&\leq |\rho|^{2(1+\delta)}_{H}K^{\frac{1+\delta}{2}}T^{\frac{1+\delta}{2}}\Big(T+\int^{t}_{s}|{u}(r)|^{2}_{H}dr\Big)^{\frac{1+\delta}{2}},
\end{align*}
where the second inequality follows from the Hypothesis $\bf H1$ with $p=2$, and the last inequality follows from the concavity of the power $\frac{1+\delta}{2}$. Using Hypothesis $\bf H3$ and inequality $(a+b)^{p}\leq 2^{p-1}(a^{p}+b^{p})$, we have
\begin{align*}
\int^{t}_{s}\Big|\int_{Z}\langle{\bf G}(r, {u}(r), i(r), z),\rho\rangle_{V}\nu(dz)\Big|^{1+\delta}dr
\leq 2^{\delta}|\rho|^{1+\delta}_{H}K^{1+\delta}\Big(T+\int^{t}_{s}|{u}(r)|^{1+\delta}_{H}dr\Big).
\end{align*}
Therefore, taking expectation and then supremum over $\ell$ on martingale terms, the above estimates imply
\begin{align*}
&\sup_{\ell}\mathbb{E}^{n_{\ell}}\int^{t}_{s}\big|\big(\rho, \sigma(r, {u}(r), i(r))Q\sigma^{\ast}(r, {u}(r), i(r))\rho\big)_{H}\big|^{1+\delta}dr
\leq C_{Q}(T)\\
& \text{and}\,\,\sup_{\ell}\mathbb{E}^{n_{\ell}}\int^{t}_{s}\Big|\int_{Z}\langle{\bf G}(r, {u}(r), i(r), z),\rho\rangle_{V}\nu(dz)\Big|dr
\leq C_{{\bf G}}(T)
\end{align*}
since, by \eqref{L^2 M}, if $\delta<1$,
\begin{align*}
&\mathbb{E}^{\mu_{n_{\ell}}}\int^{t}_{s}\|{u}(r)\|^{1+\delta}_{V}dr=\mathbb{E}\int^{t}_{s}\|{\bf u}_{n_{\ell}}(r)\|^{1+\delta}dr
\\&\leq\mathbb{E}\int^{t}_{s}\|{\bf u}_{n_{\ell}}(r)\|^{2}dr=\mathbb{E}^{\mu_{n_{\ell}}}\int^{t}_{s}\|{u}(r)\|^{2}_{V}dr\leq C.
\end{align*}

In conclusion,  the argument above shows that for $0<\delta\leq\frac{1}{5}$, there is a constant $C$ such that
$
\sup_{\ell}\mathbb{E}^{n_{\ell}}[|M^{\phi}|^{1+\delta}]\leq C
$
provided that $\mathbb{E}|{\bf u}(0)|^{3}$ is finite. Hence, we complete the proof.
\end{proof}

As $\bf M1$, $\bf M2$, and $\bf M3$ are shown, the existence theorem follows:
\begin{theorem}\label{exist}
Suppose that  $\mathbb{E}|{\bf u}_{0}|^{3}<\infty$ and ${\bf f}\in L^{3}(0, T; V')$. Then under the Hypotheses $\bf H$,
$M^{\phi}(t)$ is a $\mu$-martingale, i.e., $\mu$ is a solution to the martingale problem posed by \eqref{N-S M}. 
\end{theorem}

\subsection{Uniqueness of the solution to the regularized equation}\label{uniquely}
In this subsection, we prove that the (weak) solution obtained from Theorem \ref{exist} is pathwise unique.
\begin{theorem}\label{unique}
Let $\mathbb{E}|{\bf u}_{0}|^{3}<\infty$ and ${\bf f}\in L^{3}(0, T; V')$. Then under Hypotheses $\bf H$, the solution obtained from Theorem \ref{exist} is pathwise unique.
\end{theorem}

\begin{proof}
Let ${\bf w}={\bf u-v}$, where ${\bf u}$, ${\bf v}$ are solutions with same initial data.
Let $F(t, x, i)\colonequals e^{-\rho(t)}x$, where $\rho(t)$ is a function that will be determined later. Then the It\^o formula implies
\begin{align}\label{mg s_17}
&e^{-\rho(t)}|{\bf w}(t)|^{2}+2\nu\int^{t}_{s}e^{-\rho(s)}\|{\bf w}(s)\|^{2}ds
\\&=\int^{t}_{0}-\rho^{'}(s)e^{-\rho(s)}|{\bf w}(s)|^{2}ds\nonumber
\\&\quad-\int^{t}_{0}e^{-\rho(s)}\langle{\bf B}_{k_{\epsilon}}({\bf u}(s))-{\bf B}_{k_{\epsilon}}({\bf v}(s)),{\bf w}(s)\rangle_{V}ds\nonumber
\\&\quad+\int^{t}_{0}e^{-\rho(s)}\|\sigma(s, {\bf u}(s), \mathfrak{r}(s))-\sigma(s, {\bf v}(s), \mathfrak{r}(s))\|^{2}_{L_{Q}}ds\nonumber
\\&\quad+2\int^{t}_{0}e^{-\rho(s)}\langle{\bf w}(s), [\sigma(s, {\bf u}(s), \mathfrak{r}(s) )-\sigma(s, {\bf v}(s), \mathfrak{r}(s))]dW(s)\rangle\nonumber
\\&\quad+2\int^{t}_{0}\int_{Z}e^{-\rho(s)}\nonumber
\\&\qquad\cdot\big({\bf w}(s-), {\bf G}(s, {\bf u}(s-), \mathfrak{r}(s-), z)-{\bf G}(s, {\bf v}(s-), \mathfrak{r}(s-), z)\big)_{H}\tilde{N}_{1}\nonumber
\\&\quad+\int e^{-\rho(s)}|{\bf G}(s, {\bf u}(s-), \mathfrak{r}(s-), z)-{\bf G}(s, {\bf v}(s-), \mathfrak{r}(s-), z)|^{2}N_{1},\nonumber
\end{align}
where $\tilde{N}_{1}=\tilde{N}_{1}(dz, ds)$, $N_{1}=N_{1}(dz, ds)$, and the last integral is over $[s, t]\times Z$.

For the nonlinear term, it follows from its definition that
\begin{align*}
&|\langle{\bf B}_{k_{\epsilon}}({\bf u}(s))-{\bf B}_{k_{\epsilon}}({\bf v}(s)), {\bf w}(s)\rangle_{V}|
\\&\leq|b(k_{\epsilon}{\bf u}(s), {\bf u}(s), {\bf w}(s))-b(k_{\epsilon}{\bf v}(s), {\bf u}(s), {\bf w}(s)|
\\&\quad+|b(k_{\epsilon}{\bf v}(s), {\bf u}(s), {\bf w}(s))-b(k_{\epsilon}{\bf v}(s), {\bf v}(s), {\bf w}(s))|
\\&=|b(k_{\epsilon}{\bf w}(s), {\bf u}(s), {\bf w}(s))|+|b(k_{\epsilon}{\bf v}(s), {\bf w}(s), {\bf w}(s))|
\\&=|b(k_{\epsilon}{\bf w}(s), {\bf u}(s), {\bf w}(s))|,
\end{align*}
where the last equality follows from \eqref{vanishing of b}.
Therefore, by \eqref{b_{k} uvu} and the basic Young inequality, we see that
\begin{align*}
&\Big|\int^{t}_{0}e^{-\rho(s)}\langle{\bf B}_{k_{\epsilon}}({\bf u}(s))-{\bf B}_{k_{\epsilon}}({\bf v}(s)), {\bf w}(s)\rangle_{V}ds\Big|
\\&\leq\int^{t}_{0}e^{-\rho(s)}(\|{\bf w}(s)\|\cdot|{\bf w}(s)|\cdot\|{\bf u}(s)\|)ds
\\&\leq\nu\int^{t}_{0}e^{-\rho(s)}\|{\bf w}(s)\|^{2}ds+\frac{1}{4\nu}\int^{t}_{0}e^{-\rho(s)}|{\bf w}(s)|^{2}\|{\bf u}(s)\|^{2}ds.
\end{align*}
Therefore, choosing $\rho(t):=\frac{1}{4\nu}\int^{t}_{0}\|{\bf u}(s)\|^{2}ds$, we deduce form \eqref{mg s_17} that
\begin{align}
&e^{-\rho(t)}|{\bf w}(t)|^{2}+\nu\int^{t}_{0}e^{-\rho(s)}\|{\bf w}(s)\|^{2}ds\label{mg s_18}
\\&\leq\int^{t}_{0}e^{-\rho(s)}\|\sigma(s, {\bf u}(s), \mathfrak{r}(s))-\sigma(s, {\bf v}(s), \mathfrak{r}(s))\|^{2}_{L_{Q}}ds\nonumber
\\&\quad+2\int^{t}_{0}e^{-\rho(s)}\langle{\bf w}(s), [\sigma(s, {\bf u}(s), \mathfrak{r}(s))-\sigma(s, {\bf v}(s), \mathfrak{r}(s))]dW(s)\rangle\nonumber
\\&\quad+2\int^{t}_{0}\int_{Z}e^{-\rho(s)}\nonumber
\\&\qquad\cdot\big({\bf w}(s-), {\bf G}(s, {\bf u}(s-), \mathfrak{r}(s-), z)-{\bf G}(s, {\bf v}(s-), \mathfrak{r}(s-), z)\big)_{H}\tilde{N}_{1}\nonumber
\\&\quad+\int e^{-\rho(s)}|{\bf G}(s, {\bf u}(s-), \mathfrak{r}(s-), z)-{\bf G}(s, {\bf v}(s-), \mathfrak{r}(s-), z)|^{2}N_{1},\nonumber
\end{align}
where $\tilde{N}_{1}=\tilde{N}_{1}(dz, ds)$, $N_{1}=N_{1}(dz, ds)$, and the last integral is over $[s, t]\times Z$.
Moreover, by the Davis and the basic Young inequalities and Hypotheses $\bf H2$ and $\bf H4$, the martingale terms in \eqref{mg s_18} have the following estimates.
\begin{align*}
&\mathbb{E}\sup_{0\leq t\leq T}\int^{t}_{0}2e^{-\rho(s)}\langle{\bf w}(s), [\sigma(s, {\bf u}(s), \mathfrak{r}(s))-\sigma(s, {\bf v}(s), \mathfrak{r}(s))]dW(s)\rangle
%\\&\quad\leq 2C_{\frac{1}{2}}\mathbb{E}\Big\{\Big(\int^{T}_{0}e^{-2\rho(s)}\|(\sigma^{\ast}(s, {\bf u}(s),\mathfrak{r}(s))-\sigma^{\ast}(s, {\bf v}(s),\mathfrak{r}(s)))({\bf u}(s)-{\bf v}(s))\|^{2}_{L_{Q}}ds\Big)^{\frac{1}{2}}\Big\}
\\&\leq2\sqrt{L}C_{\frac{1}{2}}\mathbb{E}\Big\{\Big(\int^{T}_{0}e^{-2\rho(s)}|{\bf w}(s)|^{2}\cdot|{\bf w}(s)|^{2}ds\Big)^{\frac{1}{2}}\Big\}
\\&\leq2\sqrt{L}C_{\frac{1}{2}}\mathbb{E}\Big\{\sup_{0\leq t\leq T}e^{-\frac{1}{2}\rho(t)}|{\bf w}(t)|\big(\int^{T}_{0}e^{-\rho(s)}|{\bf w}(s)|^{2}\big)^{\frac{1}{2}}\Big\}
\\&\leq2\sqrt{L}C_{\frac{1}{2}}\Big\{\epsilon\sup_{0\leq t\leq T}e^{-\rho(t)}|{\bf w}(t)|^{2}+C_{\epsilon}\int^{T}_{0}e^{-\rho(s)}|{\bf w}(s)|^{2}ds\Big\}
\end{align*}
and
\begin{align*}
&\mathbb{E}\sup_{0\leq t\leq T}\int^{t}_{0}\int_{Z} 2e^{-\rho(s)}
\\&\quad\cdot\big({\bf w}(s-), {\bf G}(s, {\bf u}(s-), \mathfrak{r}(s-), z)-{\bf G}(s, {\bf v}(s-), \mathfrak{r}(s-), z)\big)_{H}\tilde{N}_{1}
%\\&\quad\leq2C_{\frac{1}{2}}\mathbb{E}\Big\{\Big(\int^{T}_{0}\int_{Z}e^{-2\rho(s)}|{\bf w}(s)|^{2}|{\bf G}(s, {\bf u}(s),\mathfrak{r}(s), z)-{\bf G}(s, {\bf v}(s),\mathfrak{r}(s), z)|^{2}\nu(dz)ds\Big)^{\frac{1}{2}}\Big\}
\\&\leq2\sqrt{L}C_{\frac{1}{2}}\mathbb{E}\Big\{\sup_{0\leq t\leq T}e^{-\frac{1}{2}\rho(t)}|{\bf w}(t)|\big(\int^{T}_{0}e^{-\rho(s)}|{\bf w}(s)|^{2}ds\big)^{\frac{1}{2}}\Big\}
\\&\leq 2\sqrt{L}C_{\frac{1}{2}}\mathbb{E}\Big\{\epsilon\sup_{0\leq t\leq T}e^{-\rho(t)}|{\bf w}(t)|^{2}+C_{\epsilon}\int^{T}_{0}e^{-\rho(s)}|{\bf w}(s)|^{2}ds\Big\},
\end{align*}
where $\tilde{N}_{1}=\tilde{N}_{1}(dz, ds)$.
As a consequence, taking supremum over $[0, T]$ and then expectation, one obtains from \eqref{mg s_18} the following.
\begin{align*}
&\mathbb{E}\sup_{0\leq t\leq T}e^{-\rho(t)}|{\bf w}(t)|^{2}+\nu\mathbb{E}\int^{T}_{0}\|{\bf w}(s)\|^{2}ds
\\&\leq 2L\mathbb{E}\int^{T}_{0}e^{\rho(s)}|{\bf w}(s)|^{2}ds+\mathfrak{C}\epsilon\mathbb{E}\sup_{0\leq t\leq T}e^{-\rho(t)}|{\bf w}(t)|^{2}
+\mathfrak{C}C_{\epsilon}\mathbb{E}\int^{T}_{0}e^{-\rho(s)}|{\bf w}(s)|^{2}ds,
\end{align*}
where $\mathfrak{C}=4\sqrt{L}C_{\frac{1}{2}}$.
Choosing $\epsilon$ small enough so that $\mathfrak{C}\epsilon<\frac{1}{2}$, one obtains from above that
\begin{align*}
\mathbb{E}\sup_{0\leq t\leq T}e^{-\rho(t)}|{\bf w}(t)|^{2}
\leq C\mathbb{E}\int^{T}_{0}e^{-\rho(s)}|{\bf w}(s)|^{2}ds\leq C\mathbb{E}\int^{T}_{0}\sup_{0\leq r\leq s}e^{-\rho(r)}|{\bf w}(r)|^2ds,
\end{align*}
where $C$ stands for a generic constant. Furthermore, we employ the Gronwall inequality to obtain
$
\mathbb{E}\sup_{0\leq t\leq T}e^{-\rho(t)}|{\bf w}(t)|^{2}\leq 0,
$
which implies the pathwise uniqueness. Hence, we complete the proof.
\end{proof}

\section{The Navier-Stokes Equation with Markov Switching}\label{original}
In this section, we prove the existence of a weak solution (in the sense of Definition \ref{Solution of MG problem}) to equation \eqref{N-S M original}.

Let ${\bf u}_{0}$ be an $H$-valued random variable such that $\mathbb{E}|{\bf u}_{0}|^{3}<\infty$. Define ${\bf u}^{\epsilon}_{0}\colonequals k_{\epsilon}{\bf u}_{0}$. Then ${\bf u}^{\epsilon}_{0}$ is an $H$-valued random variable with 
$
\mathbb{E}|{\bf u}^{\epsilon}_{0}|^{3}=\mathbb{E}|k_{\epsilon}{\bf u}_{0}|^{3}\leq\mathbb{E}|{\bf u}_{0}|^{3}<\infty.
$
Let ${\bf f}\in L^{3}(0, T; V')$. Then with the given initial data ${\bf u}^{\epsilon}_{0}$ and the external forcing $\bf f$, there exist a unique strong solution $({\bf u}^{\epsilon}(t), \mathfrak{r}(t))$ to the equation \eqref{N-S M} for each $\epsilon>0$.

Proceeding as in the argument of Proposition \ref{A priori estimates of N-S M}, we see that ${\bf u}^{\epsilon}(t)$ satisfies
\begin{align}\label{finding bound L^2 V}
\mathbb{E}\sup_{0\leq t\leq T}|{\bf u}^{\epsilon}(t)|^{2}+\nu\mathbb{E}\int^{T}_{0}\|{\bf u}^{\epsilon}(s)\|^{2}ds
\leq C_2
\end{align}
Making use of the estimate \eqref{finding bound L^2 V}, we deduce from Lemma \ref{Aldous} that the processes $\{{\bf u}^{\epsilon}\}_{\epsilon>0}$ is tight in $\mathcal{D}([0, T]; V')$.

\begin{lemma}
The sequence of processes $\{{\bf u}^{\epsilon}\}_{\epsilon>0}$ is tight in $\mathcal{D}([0, T]; V')$.
\end{lemma}

\begin{proof}
First, it is clear that the paths of the processes $\{{\bf u}^{\epsilon}\}_{\epsilon>0}$ are in $\mathcal{D}([0, T]; V')$.
Let $N>0$. Employing the Markov inequality, the estimate \eqref{finding bound L^2 V}, and the property $\|\cdot\|_{V'}<|\cdot|$, we obtain, for each rationals $t\in[0, T]$, 
\begin{align*}
\mathcal{P}\Big(\|{\bf u}^{\epsilon}(t)\|_{V'}>N\Big)\leq\frac{1}{N^{2}}\mathbb{E}\|{\bf u}^{\epsilon}(t)\|^{2}_{V'}\leq\frac{1}{N^{2}}\mathbb{E}|{\bf u}^{\epsilon}(t)|^{2}\leq\frac{C_2}{N^{2}}.
\end{align*}
Thus, for each rationals $t\in[0, T]$,
\begin{align*}
\lim_{N\rightarrow\infty}\limsup_{\epsilon\rightarrow 0}\mathcal{P}\Big(\|{\bf u}^{\epsilon}(t)\|_{V'}>N\Big)=0
\end{align*}
Let $(T_{\epsilon}, \delta_{\epsilon})$ be a sequence, where $T_{\epsilon}$ is a stopping time with $T_{\epsilon}+\delta_{\epsilon}\leq T$ and $\delta_{\epsilon}>0$ with $\delta_{\epsilon}\rightarrow 0$ as $\epsilon\rightarrow 0$. For each $N>0$, the Chebyshev inequality implies
\begin{align*}
&\mathcal{P}\Big(\|{\bf u}^{\epsilon}(T_{\epsilon}+\delta_{\epsilon})-{\bf u}^{\epsilon}(T_{\epsilon})\|_{V'}>N\Big)
\\&\leq\frac{1}{N^{2}}\mathbb{E}\|{\bf u}^{\epsilon}(T_{\epsilon}+\delta_{\epsilon})-{\bf u}^{\epsilon}(T_{\epsilon})\|^{2}_{V'}\leq\frac{1}{N^{2}}\mathbb{E}|{\bf u}^{\epsilon}(T_{\epsilon}+\delta_{\epsilon})-{\bf u}^{\epsilon}(T_{\epsilon})|^{2}.
\end{align*}
In addition, the It\^o formula and the Gronwall inequality imply that
\begin{align*}
\mathbb{E}|{\bf u}^{\epsilon}(T_{\epsilon}+\delta_{\epsilon})-{\bf u}^{\epsilon}(T_{\epsilon})|^{2}\leq\Big(\frac{1}{\nu}\mathbb{E}\int^{\delta_{\epsilon}}_{0}\|{\bf f}(s)\|^{2}_{V'}ds+2K\delta_{\epsilon}\Big)e^{2K\delta_{\epsilon}}\rightarrow 0
\end{align*}
as $\epsilon\rightarrow 0$. Therefore, $\|{\bf u}^{\epsilon}(T_{\epsilon}+\delta_{\epsilon})-{\bf u}^{\epsilon}(T_{\epsilon})\|_{V'}\rightarrow 0$ in probability as $\epsilon\rightarrow 0$. By Lemma \ref{Aldous}, the set $\{{\bf u}^{\epsilon}\}_{\epsilon>0}$ is tight in $\mathcal{D}([0, T]; V')$.
\end{proof}
The estimate \eqref{finding bound L^2 V} also shows that $\{{\bf u}^{\epsilon}\}_{\epsilon>0}$ is bounded in the space $L^{2}(
\Omega; L^{2}(0, T; V))$. Therefore, by an argument analogous to that in Proposition \ref{cpt}, we have the following proposition.
\begin{proposition}\label{u^{epsilon} to u}
The sequence of processes $\{{\bf u}^{\epsilon}\}_{\epsilon>0}$ forms a relative compact set in the space $L^{2}(\Omega; L^{2}(0, T; H))$.
\end{proposition}
Denote by ${\bf u}$ the limit (along a subsequence) of $\{{\bf u}^{\epsilon}\}_{\epsilon>0}$ in  $L^{2}(\Omega; L^{2}(0, T; H))$. The next step is to identify that ${\bf u}$ is indeed a solution to the equation \eqref{N-S M original} in the sense of Definition \ref{Solution of MG problem}.

In order to carry out the arguments similar to those in Section \ref{sec exist}, we introduce the following functions.

Let $\phi(t, i)$ be a real-valued, smooth function with compact support (in each variable). For $\rho\in\mathcal{D}({\bf A})\subseteq V$ with $\nabla\rho\in (L^{\infty}(G))^{3}$, $0\leq s\leq t$, and each element $\omega=(u, i)\in\Omega^{\dagger}$, define ${\bf M}^{\phi}(t)$ 
\begin{align}\label{M'}
&:=\phi(\langle {u}(t), \rho\rangle_{V}, i(t))-\int^{t}_{0}\sum_{j=1}^{m}\gamma_{i(r-)j}\phi(\langle {u}(r),\rho\rangle_{V}, j)dr
\\&-\int^{t}_{0}\Big(\phi^{'}(\langle {u}(r), \rho\rangle_{V}, i(r))\langle -\nu{\bf A}u(r)-{\bf B}({ u}(r))+{\bf f}(r),\rho\rangle_{V}\Big)dr\nonumber
\\&-\frac{1}{2}\int\!\!\phi^{''}\!(\langle {u}(r),\rho\rangle_{V}, i(r))\big(\rho, \sigma(r, {u}(r), i(r))Q\sigma^{*}(r, { u}(r), i(r))\rho\big)_{H}dr\nonumber
\\&-\int\Big(\phi(\langle{u}(r)+{\bf G}(r, {u}(r), i(r), z),\rho\rangle_{V}, i(r))-\phi(\langle{u}(r),\rho\rangle_{V}, i(r))\nonumber
\\&\qquad\quad-\phi^{'}(\langle{u}(r),\rho\rangle_{V}, i(r))\cdot\langle{\bf G}(r, {u}(r), i(r), z), \rho\rangle_{V}\Big)\nu_{1}(dz)dr,\nonumber
\end{align}
where the second last integral is over $[0, t]$, and the last integral is over $[0, t]\times Z$, and ${\bf M}_{\epsilon}^{\phi}(t)$ admits a similar 
expression with ${\bf B}_{k_{\epsilon}}$ replacing ${\bf B}$. 

Similar to ${\bf M2}$ in Section \ref{sec exist}, we need to prove that ${\bf M}^{\phi}_{\epsilon}(t)$ converges to ${\bf M}^{\phi}(t)$ as $\epsilon$ tends to 0.
\begin{proposition}\label{M2'}
Let ${\bf M}^{\phi}_{\epsilon}(t)$ and ${\bf M}^{\phi}(t)$ be as above. Then
\begin{align*}
\lim_{\epsilon\rightarrow 0}{\bf M}^{\phi}_{\epsilon}(t)={\bf M}^{\phi}(t).
\end{align*}
\end{proposition}
\begin{proof}
The convergence of the terms other than the nonlinear term ${\bf B}_{k_{\epsilon}}$ follows a similar argument as in the proof of assertion $\bf M2$. 

Recall the definitions of ${\bf B}$ and ${\bf B}_{k_{\epsilon}}$ from Section \ref{pre}. It suffices to show that
\begin{align*}
\lim_{\epsilon\rightarrow 0}\int^{t}_{0}b(k_{\epsilon}u(r), u(r), \rho)dr=\int^{t}_{0}b(u(r), u(r), \rho)dr.
\end{align*}

Notice that the function $u(r)$ is a generic element in the space $\Omega^{\ast}$, therefore, $u(r)\in H$. Then it follows from (2) of Lemma \ref{convolution} that
$k_{\epsilon}u\rightarrow u$ 
as $\epsilon\rightarrow 0$, in $H$. It follows from \eqref{def of b_{k}} and \eqref{negative of b} that
\begin{align*}
\int^{t}_{0}b(k_{\epsilon}u(r), u(r), \rho)dr
=-\sum_{i, j=1}^{3}\int^{t}_{0}\int_{G}(\eta_{\epsilon}\ast u)_{i}(r)\frac{\partial\rho_{j}(x)}{\partial x_{i}}u_{j}(r)dxdr.
\end{align*}
\begin{align}
&\text{Consider}\,\,
\sum_{i, j=1}^{3}\int^{t}_{0}\int_{G}(\eta_{\epsilon}\ast u)_{i}(r)\frac{\partial\rho_{j}(x)}{\partial x_{i}}u_{j}(r)dxdr\nonumber
\\&\,\,\,\,\,\,\,\,=\sum_{i, j=1}^{3}\int^{t}_{0}\int_{G}(\eta_{\epsilon}\ast u)_{i}(r)\frac{\partial\rho_{j}(x)}{\partial x_{i}}[u_{j}(r)-(\eta_{\epsilon}\ast u)_{j}(r)]dxdr\label{aux 001}
\\&\,\,\,\,\,\,\,\,\quad+\sum_{i, j=1}^{3}\int^{t}_{0}\int_{G}(\eta_{\epsilon}\ast u)_{i}(r)\frac{\partial\rho(x)_{j}}{\partial x_{i}}(\eta_{\epsilon}\ast u)_{j}(r)dxdr.\nonumber
\end{align}
The former term on the right of \eqref{aux 001} bounded by
\begin{align*}
&\|\nabla\rho\|_{L^{\infty}(G)}\int^{T}_{0}|k_{\epsilon}u(r)||u(r)-k_{\epsilon}u(r)|dr
\\&\leq\|\nabla\rho\|_{L^{\infty}(G)}\Big(\int^{T}_{0}|k_{\epsilon}u(r)|^{2}dr\Big)^{2}\Big(\int^{T}_{0}|u(r)-k_{\epsilon}u(r)|^{2}dr\Big)^{2}\rightarrow 0,
\end{align*}
as $\epsilon\rightarrow 0$ since $k_{\epsilon}u\rightarrow u$. The latter term is bounded by
\begin{align*}
\|\nabla\rho\|_{L^{\infty}(G)}\int^{T}_{0}|k_{\epsilon}u(r)|^{2}dr<C
\end{align*}
for an appropriate constant $C$, which is independent of $\epsilon$. Therefore, by the Lebesgue Dominated Convergence Theorem, the latter term converges to
\begin{align*}
\sum_{i, j=1}^{3}\int^{t}_{0}\int_{G}u_{i}(r)\frac{\partial\rho_{j}(x)}{\partial x_{i}} u_{j}(r)dxdr.
\end{align*}
Hence, we conclude that as $\epsilon \to 0$, 
\begin{align*}
&\int^{t}_{0}b(k_{\epsilon}u(r), u(r), \rho)dr=-\int^{t}_{0}b(k_{\epsilon}u(r), \rho, u(r))dr  \\
& \to -\int^{t}_{0}b(u(r), \rho, u(r))dr=\int^{t}_{0}b(u(r), u(r), \rho)dr.
\end{align*}
\end{proof}

\begin{lemma}\label{lemma conti'}
The function ${\bf M}^{\phi}(t)$ is continuous in the $(\tau^{\dagger})$-topology.
\end{lemma}

\begin{proof}
The proof of this lemma follows from the argument for proving Lemma \ref{lemma conti}. In fact, everything follows along the same lines except for the ${\bf B}$. The convergence of ${\bf B}$ follows from an argument analogous to the proof of \cite[Lem. 3.2, Ch. III]{Temam}.
\end{proof}

Denote by $\mu_{\epsilon}$ the distribution of $({\bf u}^{\epsilon}(t), \mathfrak{r}(t))$ and $\mu$ the distribution of $({\bf u}(t), \mathfrak{r}(t))$.
\begin{lemma}\label{lemma 2'}
Suppose that the Hypotheses $\bf H$ is fulfilled, $\mathbb{E}|{\bf u}_{0}|^{3}<\infty$, and ${\bf f}\in L^{3}(0, T; V')$.
There exist some $\delta>0$ such that
\begin{align*}
\sup_{\epsilon>0}\mathbb{E}^{\mu_{\epsilon}}\big[|{\bf M}^{\phi}|^{1+\delta}\big]\leq C,
\end{align*}
where $C$ is an appropriate constant.
\end{lemma}

\begin{proof}
The proof follows from the same lines as Lemma \ref{lemma 2} except for the nonlinear term ${\bf B}$. It follows from \cite[Eq. (3.74)]{Temam} that
\begin{align*}
\|{\bf B}(u(r))\|_{V'}\leq C|u(r)|^{\frac{1}{2}}\|u(r)\|^{\frac{3}{2}}
\end{align*}
for an appropriate constant $C$. Hence,
\begin{align*}
&\mathbb{E}^{\mu_{\epsilon}}\Big\{\big|\int^{t}_{s}\langle{\bf B}({u}(r)),\rho\rangle_{V}dr\big|^{1+\delta}\Big\}
\\&\leq \|\rho\|^{1+\delta}_{V}\mathbb{E}^{\mu_{\epsilon}}\Big\{\big(\int^{t}_{s}\|{u}(r)\|^{\frac{3}{2}}_{V}|{u}(r)|^{\frac{1}{2}}_{H}dr\big)^{1+\delta}\Big\}
\\&\leq \|\rho\|^{1+\delta}_{V}\Big\{\mathbb{E}^{\mu_{\epsilon}}\big(\sup_{0\leq t\leq T}|{u}(t)|^{\frac{1+\delta}{2}}_{H}\big)^{p}\Big\}^{\frac{1}{p}}\Big\{\mathbb{E}^{\mu_{\epsilon}}\big(\int^{t}_{s}\|{u}(r)\|^{\frac{3(1+\delta)}{2}}_{V}dr\big)^{q}\Big\}^{\frac{1}{q}},
\end{align*}
where $\frac{1}{p}+\frac{1}{q}=1$. Choosing $q$ such that $3(1+\delta)q=4$,  we have
\begin{align}
&\mathbb{E}^{\mu_{\epsilon}}\Big\{\big|\int^{t}_{s}\langle{\bf B}({u}(r)),\rho\rangle_{V}\big|^{1+\delta}\Big\}\label{mg s_20'}
\\&\leq \|\rho\|^{1+\delta}_{V}\Big\{\mathbb{E}^{\mu_{\epsilon}}\sup_{0\leq t\leq T}|{u}(t)|^{2(\frac{1+\delta}{1-3\delta})}_{H}\Big\}^{\frac{1-3\delta}{4}}\Big\{\mathbb{E}^{\mu_{\epsilon}}\int^{t}_{s}\|{u}(r)\|^{2}_{V}dr\Big\}^{\frac{3(1+\delta)}{4}}.\nonumber
\\&= \|\rho\|^{1+\delta}_{V}\Big\{\mathbb{E}\sup_{0\leq t\leq T}|{\bf u}^{\epsilon}(t)|^{2(\frac{1+\delta}{1-3\delta})}_{H}\Big\}^{\frac{1-3\delta}{4}}\Big\{\mathbb{E}\int^{t}_{s}\|{\bf u}^{\epsilon}(r)\|^{2}_{V}dr\Big\}^{\frac{3(1+\delta)}{4}}\nonumber
\end{align}
Taking $\delta=\frac{1}{11}$, we have $2(\frac{1+\delta}{1-3\delta})=3$. By \eqref{L^3 sup M}, we have the following estimate
\begin{align*}
\mathbb{E}\sup_{0\leq t\leq T}|{\bf u}^{\epsilon}(t)|^{3}\leq C(\mathbb{E}|{\bf u}_{0}|^{3}, \mathbb{E}\int^{T}_{0}\|{\bf f}(s)\|^{3}_{V'}, \nu, K, T).
\end{align*}
Therefore, The first expectation on the right of \eqref{mg s_20'} will have a uniform bound by \eqref{L^3 sup M} if we further assume that $\mathbb{E}|{\bf u}_{0}|^{3}<\infty$. The boundedness of $\mathbb{E}\int^{t}_{s}\|{\bf u}^{\epsilon}(r)\|^{2}_{V}dr$ follows from \eqref{L^2 M} and (3) of Lemma \ref{convolution}. Therefore, \eqref{mg s_20'} implies that, if $\delta\leq\frac{1}{11}$,
$$
\sup_{\epsilon>0}\mathbb{E}^{\mu_{\epsilon}}\Big\{\big|\int^{t}_{0}\langle{\bf B}({u}(r)),\rho\rangle_{V}\big|^{1+\delta}dr\Big\}\leq C_{{\bf B}}
\eqno$$
\end{proof}

The results of Lemmata \ref{lemma conti'} and \ref{lemma 2'} together with Lemma \ref{prop conv} imply the following proposition.
\begin{proposition}\label{M3'}
Let ${\bf u}^{\epsilon}$ be the solution \eqref{N-S M} and $\bf u$ the limit of $\{{\bf u}^{\epsilon}\}_{\epsilon>0}$ in $L^{2}(\Omega; L^{2}(0, T; H))$. Denote by $\mu_{\epsilon}$ the distribution of $({\bf u}^{\epsilon}(t), \mathfrak{r}(t))$ and $\mu$ the distribution of $({\bf u}(t), \mathfrak{r}(t))$. Then
\begin{align*}
\lim_{\epsilon\rightarrow 0}\mathbb{E}^{\mu_{\epsilon}}{\bf M}^{\phi}_{\epsilon}(t)=\mathbb{E}^{\mu}{\bf M}^{\phi}(t).
\end{align*}
\end{proposition}

Finally, we are at the stage to prove Theorem \ref{THE main theorem}.

\begin{proof}[Proof of Theorem \ref{THE main theorem}]
Given ${\bf u}_{0}$ satisfying the assumption, we define  ${\bf u}^{\epsilon}_{0}\colonequals k_{\epsilon}{\bf u}_{0}$. Then ${\bf u}^{\epsilon}_{0}$ is an $H$-valued random variable with 
$
\mathbb{E}|{\bf u}^{\epsilon}_{0}|^{3}=\mathbb{E}|k_{\epsilon}{\bf u}_{0}|^{3}\leq\mathbb{E}|{\bf u}_{0}|^{3}<\infty.
$

Given $({\bf u}_{0}, {\bf f})$ satisfying the assumption, we consider the pair of initial condition and external forcing $({\bf u}^{\epsilon}_{0}, {\bf f})$. Then this pair gives a unique strong solution to equation \eqref{N-S M} by Theorem \ref{existence and uniqueness of N-S M}, and we denote the sequence of solutions by $\{{\bf u}^{\epsilon}\}_{\epsilon>0}$.

By Proposition \ref{u^{epsilon} to u}, we see that there is a limit of ${\bf u}^{\epsilon}$ in the space
$L^{2}(\Omega; L^{2}(0, T; H))$, and we denote it by ${\bf u}$. Moreover, we denote by $\mu_{\epsilon}$ the distribution of $({\bf u}^{\epsilon}(t), \mathfrak{r}(t))$ and $\mu$ the distribution of $({\bf u}(t), \mathfrak{r}(t))$. 

Following the same argument as in Section \ref{sec exist}, it suffices to prove the following:
\begin{enumerate}
\item [$\bf M1'.$] There exists a sequence $\{\mu_{\epsilon_n}\}$ that  converges weakly to $\mu$ as $\epsilon_n \to 0$.
\item [$\bf M2'.$] $\displaystyle{\lim_{n\rightarrow \infty}}{\bf M}^{\phi}_{\epsilon_n}(t)={\bf M}^{\phi}(t)$. 
\item [$\bf M3'.$] $\displaystyle{\lim_{n\rightarrow \infty}}\mathbb{E}^{\mu_{\epsilon_n}}{\bf M}^{\phi}(t)=\mathbb{E}^{\mu}{\bf M}^{\phi}(t)$,
\end{enumerate}
where ${\bf M}^{\phi}(t)$ and ${\bf M}^{\phi}_{\epsilon_n}(t)$ are defined earlier in this section.

Now, as ${\bf M1'}, {\bf M2'}$, and ${\bf M3'}$ are direct consequences of Propositions \ref{u^{epsilon} to u}, \ref{M2'}, and \ref{M3'}, respectively, the proof is complete.
\end{proof}

\section*{Acknowledgements}
This work is a part of the PhD thesis of the first author. He thanks Professor Sundar for guidance, and Professor Xiaoliang Wan for financial support from NSF grant DMS-1622026.

%Po-Han Hsu would like to express his gratitude to his adviser Professor P. Sundar for his encouragement, helpful discussions, and guidance on this project; he appreciates Professor Xiaoliang Wan for the financial support from NSF grant DMS-1622026.


\begin{thebibliography}{99}

\bibitem{SPDE in hydrodynamic} S. Albeverio, F. Flandoli, and Y. G. Sinai:
\emph{SPDE in Hydrodynamic: recent progress and prospects}.
Lecture Notes in Mathematics, 1942. Springer-Verlag, Berlin; Fondazione C.I.M.E., Florence, 2008. 

\bibitem{BT} A. Bensoussan and R. Temam:
Equations stochastiques du type Navier-Stokes.
\emph{J. Func. Anal.}, {\bf 13} 195-222,  1973.

\bibitem{turbulence}\text{B. Birnir}: 
\emph{The Kolmogorov-Obukhov Theory of Turbulence: A Mathematical Theory of Turbulence}. 
SpringerBriefs in Mathematics, Springer, New York, 2013.

\bibitem{Billingsley} P. Billingsley: 
\emph{Convergence of Probability Measures}. 
A Wiley-Interscience Publication, Second edition, John Wiley and Sons, Inc., New York, 1999.

\bibitem{BF} F. Boyer and P. Fabrie: 
\emph{Mathematical Tools for the Study of the Incompressible Navier-Stokes Equations and Related Models}.
Applied Mathematical Sciences, 183. Springer, New York, 2013.

\bibitem{DD}  G. Da Prato and A. Debussche:
Ergodicity for the 3D stochastic Navier-Stokes equations.
\emph{J. Math. Pures et Appl.}, {\bf 82} 877-947, 2003.


\bibitem{Evans} L. C. Evans:
\emph{Partial Differential Equations}.
Second edition. Graduate Studies in Mathematics, 19. American Mathematical Society, Providence, RI, 2010.

\bibitem{FMRT} C. Foias, O. Manley, R. Rosa, and R. Temam:
\emph{Navier-Stokes Equations and and Turbulence}.
Encyclopedia of Mathematics and its applications, 83, Cambridge University Profess, Cambridge, 2001.


\bibitem{FG} F. Flandoli and D. Gatarek: 
Martingale and stationary solutions for stochastic Navier-Stokes equations. 
\emph{Prob. Th. and Rel. Fields}, {\bf 102}, 367-391, 1995.

\bibitem{FM} F. Flandoli and B. Maslowski:
Ergodicity of the 2-D Navier-Stokes equation under random perturbations. 
\emph{Comm. Math. Phys.}, {\bf 171}, 119-141, 1995.

\bibitem{Howes} N. R. Howes:
\emph{Modern Analysis and Topology}.
Universitext. Springer-Verlag, New York, 1995.

\bibitem{Sundar book} G. Kallianpur and P. Sundar: 
\emph{Stochastic Analysis and Diffusion Processes}. 
Oxford Graduate Texts in Mathematics, 24. Oxford University Press, Oxford, 2014.

%\bibitem{Khasminskii} R. Khasminskii:
%\emph{Stochastic Stability of Differential Equations}.
%Second edition, Stochastic Modelling and Applied Probability, 66 Springer, Heidelberg, 2012.

\bibitem{I-W} N. Ikeda and S. Watanabe: 
\emph{Stochastic Differential Equations and Diffusion Processes}.
North-Holland Mathematical Library, Second edition, North-Holland Publishing Co., Amsterdan, 1989. 

\bibitem{Lady} O. A. Ladyzhenskaya: 
\emph{The Mathematical Theory of Viscous Incompressible Flow}. 
Second English edition, Mathematics and its Applications, Vol. 2 Gordon and Breach, Science Publishers, New York-London-Paris 1969.

\bibitem{Leray} J. Leray:
Sur le mouvement d'un liquide visqueux emplissant l'espace. 
\emph{Acta math.} {\bf 63} (1934), 193-248.

\bibitem{Ma} J. Mattingly:
Ergodicity of the 2-D Navier-Stokes equations with random forcing and large viscosity.
\emph{Comm. Math. Phys.}, {\bf 206}, 273-288, 1999.

\bibitem{MS} J. L. Menaldi and S. S. Sritharan:
Stochastic 2-D Navier-Stokes equation.
\emph{Appl. Math. and Optim.} , {\bf 46}, 31-53, 2002.

\bibitem{Metivier} M. Metivier: 
\emph{Stochastic Partial Differential Equations in Infinite Dimensional Spaces}. 
Quaderni, Scuola Normale Superiore, Pisa, 1988.

\bibitem{O-P} W. S. O{\.z}a{\'n}ski and B. C. Pooley: 
Leray's fundamental work on the Navier-Stokes equations: a modern review of ``Sur le mouvement d'un liquide visqueux emplissant l'epsace''.
\emph{Partial Differential Equations in Fluid Dynamics}, London Math. Soc. Lecture Note Ser., 452, Cambridge Univ. Press, 2018.

\bibitem{concise} C. Pr\'ev\^ot and M. R\"ockner: 
\emph{A Concise Course on Stochastic Partial Differential Equations}. 
Lecture Notes in Mathematics, 1905, Springer, Berlin, 2007.

\bibitem{RZ} M. R\"ockner and X. Zhang:
Stochastic tamed 3D Navier-Stokes equations. 
\emph{Prob. Th. and Rel. Fields}, {\bf 145}, 211-267, 2009.

\bibitem{Sundar MHD} S. S. Sritharan and P. Sundar: 
The stochastic magneto-hydrodynamic system. 
\emph{Infin. Dimens. Anal. Quantum Probab. Relat. Top.} {\bf2} (1999), 241–265. 

\bibitem{SS} S. S. Sritharan and P. Sundar:
Large deviations for the two-dimensional Navier-Stokes equations with multiplicative noise.
\emph{Stoch. Proc. and their Appl.}, {\bf 116} 1636-1659, 2006.

\bibitem{Skorohod} A. V. Skorohod: 
\emph{Asymptotic Methods in the Theory of Stochastic Differential Equations}. 
Translations of Mathematical Monographs, 78. American Mathematical Society, Providence, RI, 1989.

\bibitem{sohr} H. Sohr: 
\emph{The Navier-Stokes Equations. An Elementary Functional Analytic Approach}.  
Modern Birkh\"auser Classics. Birkh\"auser/Springer Basel AG, Basel, 2001.

\bibitem{Temam} R. Temam:
\emph{Navier-Stokes Equations. Theory and numerical analysis}.
North-Holland Publishing Co., Amsterdam, 1984.

\bibitem{Vi} M. Viot:
\emph{Solutions faibles d'\'equations aux deriv\'ees partielles stochastique non lineaires}.
Th\`ese, Univ. Pierre et Marie Curie, Paris, 1976.

\bibitem{VF} Vishik and Fursikov:
\emph{Mathematical Problems in Statistical Hydromechanics}.
Kluwer Academic Publ., Boston, 1988.


\bibitem{Walsh} J. B. Walsh:
\emph{An Introduction to Stochastic Partial Differential Equations.}
Lecture Notes in Math., 1180, Springer, Berlin, 1986.

\bibitem{Y-W} T. Yamada and S. Watanabe: 
On the uniqueness of solutions of stochastic differential equations. 
\emph{J. Math. Kyoto Univ.} {\bf 11} (1971), 155-167.


%\bibitem{Yin and Zhu} G. Yin and C. Zhu:
%\emph{Hybrid Switching Diffusions: Properties and Applications}.
%Stochastic Modelling and Applied Probability, 63. Springer, New York, 2010.

\end{thebibliography}
\end{document}